\documentclass[11pt, twoside]{article}
\usepackage{amsfonts}

\usepackage{amssymb}
\usepackage{amsmath}
\usepackage{amsthm}
\usepackage{xcolor}
\usepackage{mathrsfs}

\allowdisplaybreaks

\allowdisplaybreaks

\def\rr{{\mathbb R}}
\def\rn{{{\rr}^n}}

\def\nn{{\mathbb N}}
\def\fz{\infty}

\def\cs{{\mathcal S}}

\def\az{\alpha}

\renewcommand\hat{\widehat}

\def\vz{\varphi}

\def\wz{\widetilde}

\def\hs{\hspace{0.3cm}}

\def\r{\right}
\def\lf{\left}

\def\bint{{\ifinner\rlap{\bf\kern.30em--}
\int\else\rlap{\bf\kern.35em--}\int\fi}\ignorespaces}

\def\sbint{{\ifinner\rlap{\bf\kern.32em--}
\hspace{0.078cm}\int\else\rlap{\bf\kern.45em--}\int\fi}\ignorespaces}

\newtheorem{theorem}{Theorem}[section]
\newtheorem{lemma}[theorem]{Lemma}

\newtheorem{proposition}[theorem]{Proposition}

\theoremstyle{definition}
\newtheorem{remark}[theorem]{Remark}
\newtheorem{definition}[theorem]{Definition}
\numberwithin{equation}{section}

\numberwithin{equation}{section}

\numberwithin{equation}{section}

\begin{document}

\arraycolsep=1pt

\title{\Large\bf Maximal Function Characterizations of Hardy Spaces on ${\mathbb{R}}^{n}$ with Pointwise Variable Anisotropy \footnotetext{\hspace{-0.35cm} {\it 2010
Mathematics Subject Classification}. {42B30, 42B25.}
\endgraf{\it Key words and phrases.} anisotropy, Hardy space, continuous ellipsoid cover,  maximal function.
\endgraf This project is supported by the National Natural Science
Foundation of China (Grant Nos. 11861062  \& 11561065) and the Natural Science Foundation of Xinjiang
Uygur Autonomous Region (Grant Nos. 2019D01C049, 62008031 \& 042312023).
\endgraf $^\ast$\,Corresponding author.
}}
\author{Aiting Wang, Wenhua Wang, Xinping Wang and Baode Li$^\ast$}
\date{ }
\maketitle

\vspace{-0.8cm}

\begin{center}
\begin{minipage}{13cm}\small
{\noindent{\bf Abstract} \quad
In 2011, Dekel et al. developed highly geometric Hardy spaces $H^p(\Theta)$, for the full range $0<p\leq 1$, which are constructed by continuous multi-level ellipsoid covers $\Theta$ of $\mathbb{R}^n$ with high anisotropy in the sense that the ellipsoids can change shape rapidly from point to point and from level to level. In this article,  the authors further
obtain some real-variable characterizations of $H^p(\Theta)$ in terms of the radial, the non-tangential and the tangential maximal functions, which generalize the known results on the anisotropic Hardy spaces of Bownik.}
\end{minipage}
\end{center}


\section{Introduction\label{s1}}
As a generalization of the classical isotropic Hardy spaces $H^{p}(\mathbb{R}^n)$ \cite{fe}, anisotropic Hardy spaces $H_{A}^{p}(\mathbb{R}^n)$ were introduced and investigated by Bownik \cite{b} in 2003. These spaces were defined on $\rn$ associated with a fixed expansive matrix which act on ellipsoid instead of the Euclidean balls. In \cite{bb,bd,blyz,h,w,zl}, many authors also studied Bownik's anisotropic Hardy spaces. In 2011, Dekel et al. further \cite{dpw} generalized Bownik's spaces by constructing
Hardy spaces with pointwise variable anisotropy $H^p(\Theta),\,0<p\le 1$, associated with an ellipsoid cover $\Theta$.
The anisotropy in Bownik's Hardy spaces is the same one in each point in $\rn$, while the anisotropy in $H^p(\Theta)$ can change rapidly from point to point and from level to level. Moreover, the ellipsoid cover $\Theta$ is a very general setting which includes the classical isotropic setting, non-isotropic setting of Calder\'on and Torchinsky \cite{CT} and the anisotropic
setting of Bownik \cite{b} as special cases, see more details in \cite[pp.\,2-3]{b}
and \cite[p.\,157]{ddp}.

On the other hand, maximal function characterizations are very fundamental characterizations of Hardy spaces and they are crucial to conveniently apply the
real-variable theory of Hardy spaces $H^{p}(\mathbb{R}^n)$ with $p\in(0,\,1]$. Maximal functions characterizations was first shown for the classical isotropic Hardy spaces $H^{p}(\mathbb{R}^n)$ by Fefferman and Stein in their fundamental work \cite{fe}, \cite[Chapter III]{s}. Analogous results were shown by Calder$\acute{\text{o}}$n and Torchinsky \cite{CT, ct} for parabolic $H^{p}$ spaces and Uchiyama \cite{u} for $H^{p}$ on the space of homogeneous type. In 2003, Bownik \cite[p.42]{b} obtained the maximal function characterizations of the anisotropic Hardy space $H_{A}^{p}(\mathbb{R}^n)$.
Motivated by the above mentioned facts, a natural question arises: Do anisotropic Hardy spaces $H^p(\Theta)$ have maximal function characterizations\,$?$ In this article, we shall answer the problem affirmatively.

This article is organized as follows.

In Section \ref{s2}, we recall some notation and definitions concerning
anisotropic continuous ellipsoid cover $\Theta$, several maximal functions, anisotropic Hardy spaces $H^p(\Theta)$ defined via the grand radial maximal function. We also give some propositions about $H^p(\Theta)$, several classes of variable anisotropic maximal functions and Schwartz functions since they provide tools for further work. In Section 3, we first state main result: if ellipsoid cover $\Theta$ is pointwise continuous (see Definition \ref{d2.1}), we may obtain
some real-variable characterizations of $H^p(\Theta)$ in terms of the radial, the non-tangential and the tangential maximal functions (see Theorem \ref{t4.1}). Then we present several lemmas which are isotropic extensions in the setting of variable anisotropy and finally we show the proof of main result.

In the process of the proof of main result, the proof methods is from Stein \cite{fe} and Bownik \cite{b}. However, it is worth pointing out that the
ellipsoids of Bownik changes with the level $k\in \mathbb{Z}$,
whereas in our case the ellipsoids of Dekel changes with not only the the level $t\in \mathbb{R}$ but also with the center of the ellipsoid $x\in \mathbb{R}^n$. This makes the proof complicated and needs many subtle estimates such as Propositions \ref{l4.4}, \ref{l4.5} and Lemma \ref{l4.6}.

Finally, we make some conventions on notation.
Let $\mathbb{N}_0:=\{0,\,1,\, 2,\,\ldots\}$.
For any $\az:=(\az_1,\ldots,\az_n)\in\nn^n_0$,
$|\az|:=\az_1+\cdots+\az_n$ and $\partial^\az:=
(\frac{\partial}{\partial x_1})^{\az_1}\cdots
(\frac{\partial}{\partial x_n})^{\az_n}$.
Throughout the whole paper, we denote by $C$ a {positive
constant} which is independent of the main parameters, but it may vary from line to line. For any sets $E,\,F \subset \rn$, we use $E^\complement$ to denote the set $\rn\setminus E$. If there are no special instructions, any space $\mathcal{X}(\rn)$ is denoted simply by $\mathcal{X}$. Denote by $\cs$ the space of all Schwartz functions, $\cs'$ the space of all tempered distributions.




\section{Preliminary and Some Basic Propositions\label{s2}}
In this section, we first recall the notion of continuous ellipsoid covers $\Theta$ and we introduce the pointwise continuity for $\Theta$. An {\it ellipsoid} $\xi$ in $\rn$ is an image of the Euclidean unit ball $\mathbb{B}^{n}:=\{x\in\rn: |x|<1\}$ under an {affine} {transform}, i.e.
$$\xi:=M_\xi(\mathbb{B}^{n})+c_\xi,$$ where $M_\xi$ is a nonsingular matrix and $c_\xi \in \mathbb{R}^{n}$ is the center.

Let us begin with the definition of continuous ellipsoid covers, which was introduced in \cite[Definition 2.4]{ddp}.

\begin{definition}\label{d2.1}
We say that
$$\Theta:=\{ \theta(x,\,t): x\in\rn,t\in\mathbb{R}\}$$
is a {\it continuous ellipsoid cover} of $\rn$, or shortly an {\it ellipsoid cover},
if there exist positive constants $p(\Theta):=\{a_1,\ldots, a_6\}$, such that:
\begin{itemize}
\item[(i)]
For every $x\in \rn$ and $t\in \mathbb{R}$, there exists an ellipsoid $\theta(x,\,t):=M_{x,\,t}(\mathbb{B}^{n})+x$ satisfying
\begin{equation}\label{e2.1}
a_12^{-t}\leq|\theta(x,\,t)|\leq a_2 2^{-t}.
\end{equation}
\item[(ii)]
Intersecting ellipsoids from $\Theta$ satisfy a ``shape condition'', i.e., for any $x,\,y\in \rn$, $t\in \mathbb{R}$ and $s\ge0$, if $\theta(x,\,t)\cap \theta({y,\,t+s})\ne\emptyset $, then

\begin{equation}\label{e2.2}
a_3 2^{-a_4 s}\le \frac{1}{\| (M_{y,\,t+s})^{-1} M_{x,\,t}\|}
\le \|(M_{x,\,t})^{-1} M_{y,\,t+s}\|
\le a_5 2^{-a_6 s}.
\end{equation}
\end{itemize}
Here, $\|\cdot \|$ is the matrix norm given by $\|M\|:=\max_{|x|=1}|Mx|$ for an $n\times n$ real matrix $M$.

The word {\it continuous} refers to the fact that ellipsoids $\theta_{x,t}$ are defined for all values of $x\in \rn$ and $t\in \mathbb{R}$. And we say that a continuous ellipsoid cover
$\Theta$ is {\it pointwise continuous} if for every $t\in \mathbb{R}$, the matrix valued function $x \mapsto M_{x,t}$ is continuous. That is,
\begin{align}\label{ef2.3}
\|M_{x',t}-M_{x,t}\|\rightarrow 0 \ \ \text{as} \ \ x'\rightarrow x.
\end{align}
\end{definition}

\begin{remark}\label{r2.2s}
By \cite[Theorem 2.2]{bll}, we know that the pointwise continuous assumption is not necessary since it is always possible to construct an equivalent ellipsoid cover
\[
\Xi:=\{ \xi_{x,\,t}: x\in\rn,t\in\mathbb{R}\}
\]
such that $\Xi$ is pointwise continuous and $\Xi$ is equivalent to $\Theta$. Here we say that two ellipsoid covers $\Theta$ and $\Xi$ are {\it equivalent} if there exists a constant $C>0$ such that for any  $x\in\rn$ and $t\in\mathbb {R}$, we have
$$\frac1{C} \xi_{x,t} \subset \theta_{x,t} \subset C \xi_{x,t}.$$
\end{remark}

Taking $M_{y,\,t+s}=M_{x,\,t}$ in \eqref{e2.2}, we have
\begin{align}\label{ef2.4}
a_{3}\leq1 \ \ \text{and}\ \ a_{5}\geq1.
\end{align}
More properties about ellipsoid covers, see \cite{ddp,dpw}.

 For any $N,\,{\widetilde{N}}\in\mathbb{N}_0$ with $N\le{\widetilde{N}}$, let
\begin{eqnarray*}
\cs_{N,\,{\widetilde{N}}}:=\lf\{\psi\in\cs:\ \lf\|\psi \r\|_{\mathcal{S}_{N,\,{\widetilde{N}}}}:=\max_{\az\in\nn^n_0,\,|\az|\le N}\sup_{y\in\rn}(1+|y|)^{{\widetilde{N}}}
|\partial^\az\psi(y)|\le 1\r\}
\end{eqnarray*}

Let $\Theta$ be an ellipsoid cover. For any $\varphi \in \mathcal{S}$, $x\in\rn$, $t\in\mathbb{R}$ and $\theta(x,\,t)=M_{x,\,t}(\mathbb{B}^{n})+x$, denote
$$\varphi_{x,\,t}(y):=\lf|{\rm det}(M^{-1}_{x,t})\r|\varphi(M^{-1}_{x,\,t}y),\ \ \ y\in\rn.$$

Now, we give the notions of anisotropic variants of the non-tangential, the grand non-tangential, the radial, the grand radial and the tangential maximal functions.

\begin{definition}\label{d3.2}
Let $f\in\mathcal{S'}$, $\varphi\in\mathcal{S}$ and $N,\,\wz{N}\in\mathbb{N}_0$ with $N\le\wz{N}$. We define the {\it non-tangential}, the {\it grand non-tangential}, the {\it radial}, the {\it rand radial} and the {\it tangential maximal functions}, respectively as
$$M_\varphi f(x):=\sup_{t\in\mathbb{R}}\sup_{y\in{\theta(x,\,t)}}|f\ast\varphi_{x,\,t}(y)|,\ \ \ x\in\rn,$$

$$M_{N,\,\wz{N}} f(x):=\sup_{\varphi\in\mathcal{S}_{N,\,\wz{N}}}M_\varphi f(x),\ \ \ x\in\rn. $$

$$M^0_\varphi f(x):=\sup_{t\in\mathbb{R}}|f\ast\varphi_{x,\,t}(x)|,\ \ \ x\in\rn,$$

$$M^0_{N,\,\wz{N}}f(x):=\sup_{\varphi \,\in\mathcal{S}_{N,\,\wz{N}}}M^\circ_\varphi f(x),\ \ \ x\in\rn.$$

$$T^N_\varphi f(x):=\sup_{t\in\mathbb{R}}\sup_{y\in{\mathbb{R}}^n}|f\ast\varphi_{x,\,t}(y)|\lf(1+\lf|M^{-1}_{x,\,t}(x-y)\r|\r)^{-N},\ \ \ x\in\rn.$$
\end{definition}

\begin{remark}\label{r3.5}
It is immediate that we have the following pointwise estimate among the radial, the non-tangential and the tangential maximal functions:
$$M^0_\varphi f(x)\leq M_\varphi f(x)\leq 2^{N}\,T^N_\varphi f(x),\ \ \ x\in\rn.$$
\end{remark}

Next, we recall the definition of Hardy spaces with pointwise variable anisotropy \cite[Definition 3.6]{dpw} via the grand radial maximal function.

Let $\Theta$ be an ellipsoid cover of $\rn$ with parameters $p(\Theta)=\{a_1,\cdots,a_6\}$
and $0<p\le 1$. We define $N_p(\Theta)$ as the minimal integer satisfying
\begin{equation}\label{e2.3}
N_p:=N_p(\Theta)>\frac{\max(1,a_4)n+1}{a_6p},
\end{equation}
and then $\wz{N}_p(\Theta)$ as the minimal integer satisfying
\begin{equation}\label{e2.4}
\wz{N}_p:=\wz{N}_p(\Theta)>\frac{a_4N_p(\Theta)+1}{a_6}.
\end{equation}

\begin{definition}\label{d3.7}
Let $\Theta$ be an ellipsoid cover and $0<p\le 1$. Define $M^{0}:=M^{0}_{N_p,\wz{N}_p}$ and the {\it anisotropic Hardy space} is defined as
$$H_{N_p,\,\wz{N}_p}^p(\Theta):=\{f\in\mathcal{S'}:\,M^{0} f\in L^p\}$$
with the (quasi-)norm $\|f\|_{H^p(\Theta)}:=\|M^{0} f\|_{L^{p}}$.
\end{definition}

\begin{remark}
By Remark \ref{r2.2s}, we know that for every continuous ellipsoid cover $\Theta$, it exists an equivalent pointwise continuous ellipsoid cover $\Xi$. This implies that their corresponding (quasi-)norms $\rho_\Theta(\cdot,\cdot)$ and $\rho_\Xi(\cdot,\cdot)$ are also equivalent and hence the corresponding Hardy spaces $H^p(\Theta)=H^p(\Xi)(0<p\leq 1)$ with equivalent (quasi-)norms (see \cite[Theorem 5.8]{dpw}). So here and hereafter we always consider $\Theta$ of $H^p(\Theta)$ to be a pointwise continuous ellipsoid cover.
\end{remark}

\begin{proposition}{\label{p2.7}}
Let $\Theta$ be an ellipsoid cover, $0<p\le 1\le q\le\fz$, $p<q$ and $l\geq N_p$ with $N_p$ as in \eqref{e2.3}. If $N\geq N_p$ and $\wz{N}\geq ({a_4N+1})/{a_6}$, then
 \begin{align*}
H_{N_p,\,\wz{N}_p}^p(\Theta)=H^p_{q,\,l}(\Theta)=H_{N,\,\wz{N}}^p(\Theta)
\end{align*}
with equivalent (quasi-)norms, where $H^p_{q,\,l}(\Theta)$ denotes  the atomic Hardy space with pointwise variable anisotropy; see \cite[Definition 4.2]{dpw}.
\end{proposition}
\begin{proof}
This proposition is a corollary of \cite[Theorems 4.4 and 4.19]{dpw}. Indeed, by Definition \ref{d3.7}, we obtain that, for any $N\geq N_p$ and $\wz{N}\geq ({a_4N+1})/{a_6}$,
$$H_{N_p,\,\wz{N}_p}^p(\Theta)\subseteq H_{N,\,\wz{N}}^p(\Theta).$$
Combining this and $H^p_{q,\,l}(\Theta)\subseteq H_{N_p,\,\wz{N}_p}^p(\Theta)$ (see \cite[Theorem 4.4]{dpw}), we obtain
\begin{align}\label{ea2.5}
H^p_{q,\,l}(\Theta)\subseteq H_{N,\,\wz{N}}^p(\Theta).
\end{align}
By checking the definition of anisotropic $(p,q,l)$-atom (see \cite[Definition 4.1]{dpw}), we know that every $(p,\,\fz,\,l)$-atom is also a $(p,\,q,\,l)$-atom and hence
$$H^p_{\fz,\,l}(\Theta)\subseteq H^p_{q,\,l}(\Theta).$$
Let $l'\geq \max(l,\,{N})$. By a similar argument to the proof of \cite[Theorem 4.19]{dpw}, we obtain
\begin{align*}
H_{N,\,\wz{N}}^p(\Theta)\subseteq H^p_{\fz,\,l'}(\Theta),
\end{align*}
where $N\geq N_p$ and $\wz{N}\geq ({a_4N+1})/{a_6}$.
Thus,
\begin{align}\label{ea2.6}
H_{N,\,\wz{N}}^p(\Theta)\subseteq H^p_{\fz,\,l'}(\Theta)\subseteq H^p_{\fz,\,l}(\Theta)\subseteq H^p_{q,\,l}(\Theta).
\end{align}
Combining \eqref{ea2.5} and \eqref{ea2.6}, we conclude that
\begin{align*}
H_{N_p,\,\wz{N}_p}^p(\Theta)=H^p_{q,\,l}(\Theta)=H_{N,\,\wz{N}}^p(\Theta)
\end{align*}
with equivalent (quasi-)norms.
\end{proof}

\begin{remark}\label{r2.4}
From Proposition \ref{p2.7},
we deduce that, for any integers $N\geq N_p$ and $\wz{N}\geq ({a_4N+1})/{a_6}$, the definition of
$H_{N,\,\wz{N}}^p(\Theta)$ is independent of $N$ and $\wz{N}$.
Therefore, from now on, we denote $H_{N,\,\wz{N}}^p(\Theta)$ with $N\geq N_p$ and $\wz{N}\geq ({a_4N+1})/{a_6}$ simply by $H^p(\Theta)$.
\end{remark}
\begin{proposition}\label{l2.2}{\rm\cite[Lemma 2.3]{dpw}}
 Let $\Theta$ be an ellipsoid cover. Then there exists a constant $J:=J(p(\Theta))\geq1$ such that for any $x\in\rn$ and $t\in\mathbb{R}$,
 $$ 2M_{x,\,t}(\mathbb{B})+x\subset\theta(x,t-J).$$
Here and hereafter, let $J$ always be as in Proposition \ref{l2.2}.
\end{proposition}

\begin{definition}\label{d3.6}{\rm\cite[Definition 3.1]{dpw}}
Let $\Theta$ be an ellipsoid cover. For any locally integrable function $f$, \emph{maximal function of Hardy-Littlewood type} of $f$ is defined by
$$M_{\Theta}f(x):=\sup_{t\in\mathbb{R}}\frac{1}{|\theta(x,t)|}\int_{\theta(x,t)}|f(y)|\,dy,\ \ \ x\in\rn.$$
\end{definition}

\begin{proposition}\label{l4.3}{\rm\cite[Theorem 3.3]{dpw}}
Let $\Theta$ be an ellipsoid cover. Then
\begin{enumerate}
\item[(i)] there exists a constant $C$ depending only on $p(\Theta)$ and $n$ such that for all $f\in\,L^{1}$ and $\alpha>0$,
 \begin{align}\label{e4.2}
|\{x:\,M_{\Theta}f(x)>\alpha\}|\leq C {\alpha}^{-1}\|f\|_{L^{1}};
\end{align}
\item[(ii)] for $1<p< \fz$ there exists a constant $C_p$ depending only on $C$ and $p$ such that for all $f\in\,L^{p}$,
 \begin{align}\label{e4.3}
\|M_{\Theta}f\|_{L^{p}}\leq C_{p} \|f\|_{L^{p}}.
\end{align}
\end{enumerate}
\end{proposition}

We give some useful results about variable anisotropic maximal functions with different apertures. They also play important roles to obtain the maximal function characterizations of $H^p(\Theta)$. For any given $x\in \rn$, suppose that $F_{x}:\mathbb{R}^{n}\times\mathbb{R}\rightarrow{(0,\,\fz)}$
is a Lebesgue measurable function which satisfies that for any given $y\in \rn$ and $t\in \mathbb{R}$, $F_{x}(y,t)$ is continuous with respect to the variable $x$. For fixed $l\in \mathbb{Z}$ and $t_{0}<0$
define the maximal function of $F_x$ with aperture $l$ as
\begin{align}\label{e3.3x}
F_{l}^{*\,t_{0}}(x):=\sup_{t\geq t_{0}}\sup_{y\in\theta(x,\,t-{l}J)}F_{x}(y,t).
\end{align}
\begin{proposition}{\rm\label{p2.12}}
For any $l\in \mathbb{Z}$ and $t_{0}<0$, let $F_{l}^{*\,t_{0}}$ be as in \eqref{e3.3x}. If the ellipsoid cover $\Theta$ is pointwise continuous, then $F_{l}^{*\,t_{0}}: \mathbb{R}^n\rightarrow(0, \infty] $ is lower semicontinuous, i.e., $$\{x\in\rn: F_{l}^{*\,t_{0}}(x) > \lambda\} \ \ \text{is open for any}\ \ \lambda>0.$$
\end{proposition}
\begin{proof}
If $F_{l}^{*\,t_{0}}(x)>\lambda$ for some $x\in\rn$, then there exist $t\geq t_{0}$ and $y\in\theta(x,\,t-{l}J)$ so that $F_x(y,t)>\lambda$. Since $\theta({x,t})$ is continuous for variable $x$ (see Remark \ref{r2.2s}),  there exists $\delta_1>0$ so that for any $x'\in U(x,\delta_1):=\{z\in\rn:|z-x|<\delta_1\}$,
 $y\in \theta(x',\,t-{l}J)$.

 Besides, for $\mu:=(F_x(y,t)+\lambda)/(2\lambda)>1$, from $F_x(y,t)>\lambda$, it follows that
 $$F_x(y,t)>\mu\lambda.$$
 For any given $y\in \rn$ and $t\in \mathbb{R}$, since $F_{x}(y,t)$ is continuous for variable $x$, taking $\epsilon:=[1-(1/\mu)]F_x(y,t)$, there exists $\delta_2>0$ so that for any $x'\in U(x,\delta_2)$,
$$|F_{x}(y,t)-F_{x'}(y,t)|<\epsilon$$
and hence $F_{x'}(y,t)>(1/\mu)F_{x}(y,t)$, which together with $F_x(y,t)>\mu\lambda$ implies that
$$F_{x'}(y,t)>\lambda.$$

Taking $\delta:=\min\{\delta_1,\delta_2\}$, for any $x'\in U(x,\delta)$, we have
$$y\in \theta(x',\,t-{l}J)$$\ \ \text{and}\ \ \ $$F_{x'}(y,t)>\lambda,$$
which implies $F_{l}^{*\,t_{0}}(x')>\lambda$. Thus, $\{x\in\rn: F_{l}^{*\,t_{0}}(x) > \lambda\}$ is open.
\end{proof}

By Proposition \ref{p2.12}, we obtain that $\{x\in\rn: F_{l}^{*\,t_{0}}(x) > \lambda\}$ is Lebesgue measurable. Based on this and inspired by \cite[Lemma 7.2]{b}, the following Proposition \ref{l4.4} shows some estimates for maximal function $F_{l}^{*\,t_{0}}$.

\begin{proposition}\label{l4.4}
Let $K>0$, $\mathbb{B}_{K}:=\{y\in\rn:|y|<K\}$, $\Theta$ be an ellipsoid cover and let $F_{l}^{*\,t_{0}}$ and $F_{l'}^{*\,t_{0}}$ be as in \eqref{e3.3x} with
 integers $l>l'$ and $t_0<0$.  Then there exists a constant $C>0$ which depends on parameters $p(\Theta)$ such that for any functions $F_{l}^{*\,t_{0}}$, $F_{l'}^{*\,t_{0}}$ and $\lambda>0$, we have
\begin{align}\label{ea4.3}
\lf|\lf\{x\in \mathbb{B}_{K}:F_{l}^{*\,t_{0}}(x)>\lambda\r\}\r|\leq{C}2^{(l-{l'})J}\lf|\lf\{x\in \mathbb{B}_{K}:F_{l'}^{*\,t_{0}}(x)>\lambda\r\}\r|
\end{align}
and
\begin{align}\label{e4.4}
\int_{\mathbb{B}_{K}}F_{l}^{*\,t_{0}}(x)\,dx\leq {C}2^{(l-{l'})J}\int_{\mathbb{B}_{K}} F_{l'}^{*\,t_{0}}(x)\,dx.
\end{align}
\end{proposition}
\begin{proof}
Let $\Omega:=\{x\in \mathbb{B}_{K}:F_{l'}^{*\,t_{0}}(x)>\lambda\}$. We claim that
 \begin{align}\label{e2.15x}
\lf\{x\in \mathbb{B}_{K}:F_{l}^{*\,t_{0}}(x)>\lambda\r\}\subset\lf\{x\in\rn:M_{\Theta}({\chi_{\Omega}})(x)\geq C_{1}2^{(l'-l)J}\r\},
\end{align}
where $C_1$ is a positive constant to be fixed later. Assuming the claim holds for the moment, from this and weak type (1,1) of $M_\Theta$ (see \eqref{e4.2}), we deduce
 \begin{align*}
\lf|\lf\{x\in \mathbb{B}_{K}:F_{l}^{*\,t_{0}}(x)>\lambda\r\}\r|&\leq\lf|\lf\{x\in\rn:M_{\Theta}({\chi_{\Omega}})(x)\geq C_{1}2^{(l'-l)J}\r\}\r|\\
&\leq C_{1}^{-1}2^{(l-l')J}\|\chi_{\Omega}\|_{L^{1}}\leq C2^{(l-l')J}|\Omega|
 \end{align*}
 and hence \eqref{ea4.3} holds true,  where $C:=1/C_{1}.$  Furthermore, integrating \eqref{ea4.3} on $(0,\,\fz)$ with respect to $\lambda$ yields \eqref{e4.4}. Therefore, it remains to show \eqref{e2.15x}.

Suppose $F_{l}^{*\,t_{0}}(x)>\lambda$ for some $x\in \mathbb{B}_{K}$. Then
there exist $t$ with $t\geq t_{0}$ and $y\in\theta(x,\,t-{l}J)$ such that $F_x(y,\,t)>\lambda$. For any $l,\,{l'}\in\mathbb{Z}$ and $l\geq{l'}$, we first prove the following holds true:
 \begin{align}\label{e2.15}
 {a_{5}}^{-1}\,\theta(y,\,t-{l'}J)\subseteq \theta(x,\,t-(l+1)J)\cap\Omega.
 \end{align}
For any $z\in{a_{5}}^{-1}\,\theta(y,\,t-{l'}J)$, by \eqref{ef2.4}, we have
$z\in \theta(y,\,t-{l'}J)$
and hence $$\theta(z,\,t-{l'}J)\cap \theta(y,\,t-{l'}J)\neq \emptyset.$$
Thus, by \eqref{e2.2}, we have
$$\lf\|M_{z,\,t-l'J}^{-1}\,M_{y,\,t-l'J}\r\|\leq a_{5}.$$
From this, it follows that
$${a_{5}}^{-1}M_{z,\,t-l'J}^{-1}\,M_{y,\,t-l'J}(\mathbb{B}^{n})\subseteq \mathbb{B}^{n}$$
and hence
 \begin{align*}
{a_{5}}^{-1}\,M_{y,\,t-l'J}(\mathbb{B}^{n})\subseteq M_{z,\,t-l'J} (\mathbb{B}^{n}).
\end{align*}
By this and $y\in {a_{5}^{-1}}\,M_{y,\,t-l'J}(\mathbb{B}^{n})+z$, we obtain $y\in\theta(z,\,t-{l'}J)$.
From this and $F_x(y,t)>\lambda$ with $t\geq t_{0}$, we deduce that $F_{l'}^{*\,t_{0}}(z)>\lambda$ and hence $z\in \Omega$, which implies
\begin{align}\label{eb4.6}
{a_{5}^{-1}}\,\theta(y,\,t-{l'}J)\subseteq \Omega.
\end{align}

Besides, by $y\in\theta(x,\,t-{l}J)$, \eqref{e2.2} and $l\geq{l'}$, we have
$$\lf\|M_{x,\,t-lJ}^{-1}\,M_{y,\,t-l'J}\r\|\leq a_{5}2^{-a_{6}(l-l')J}\leq a_{5}.$$
From this, it follows that
$${a_{5}}^{-1}M_{x,\,t-lJ}^{-1}\,M_{y,\,t-l'J}(\mathbb{B}^{n})\subseteq \mathbb{B}^{n}$$
and hence
\begin{align*}
{a_{5}}^{-1}\,M_{y,\,t-l'J}(\mathbb{B}^{n})\subseteq M_{x,\,t-lJ} (\mathbb{B}^{n}).
\end{align*}
By this, \eqref{ef2.4}, $y\in\theta(x,\,t-{l}J)$ and Proposition \ref{l2.2}, we obtain
\begin{align*}
{a_{5}}^{-1}M_{y,\,t-l'J}(\mathbb{B}^{n})+y
\subseteq 2M_{x,\,t-lJ}(\mathbb{B}^{n})+x
\subseteq\theta(x,\,t-(l+1)J).
\end{align*}
From this and \eqref{eb4.6}, we deduce that \eqref{e2.15} holds true.

Next, let's prove \eqref{e2.15x}. By \eqref{e2.15} and \eqref{e2.1}, we obtain
\begin{align}\label{e4.6}
|\theta(x,\,t-(l+1)J)\cap \Omega|&\geq {(a_{5})}^{-n}|\theta(y,\,t-l'J)|\\
&\geq\frac{a_{1}}{(a_{5})^{n}}2^{l'J-t}\nonumber.
\end{align}
Taking $b_{0}:=t-(l+1)J$, by \eqref{e2.1} and \eqref{e4.6}, we have
\begin{align*}
\frac{1}{|\theta(x,\,b_{0})|}\int_{\theta(x,\,b_{0})}|\chi_{\Omega}(y)|dy&\ge{a_{2}}^{-1}2^{b_{0}}|\theta(x,\,b_{0})\cap\Omega|
\geq\frac{a_{1}}{(a_{5})^{n}a_{2}}2^{(l'-l-1)J},
\end{align*}
 which implies $M_{\Theta}(\chi_{\Omega})(x)\geq C_{1}2^{(l'-l)J}$ and hence \eqref{e2.15x} holds true,
  where\\ $C_{1}:=2^{-J}{a_{1}}/[{(a_{5})^{n}a_{2}}]$.
\end{proof}

The following result enables us to pass from one function in $\mathcal{S}$ to the sum of dilates of another function in $\mathcal{S}$ with nonzero mean, which is a variable anisotropic extension of \cite[ p.\,93,\, Lemma 2]{s} of Stein and \cite[Lemma 7.3 ]{b} of Bownik.

\begin{proposition}\label{l4.5}
Suppose $\Theta$ is an ellipsoid cover, $\varphi\in \mathcal{S}$ and $\int_{\mathbb{R}^{n}}\varphi(x)\,dx\neq0$. For every $\psi \in \mathcal{S}$, $x\in\mathbb{R}^n$ and $t\in\mathbb{R} $ there exists a sequence of test functions $\{\eta^{k}\}_{k=0}^{\fz}$, $\eta^{k}\in \mathcal{S}$ such that
\begin{align}\label{e4.7}
\psi(\xi)=\sum_{k=0}^{\fz}\eta^{k}*\widetilde{\varphi}^{k}(\xi) \ \ \text{converges in}\ \ \mathcal{S},
\end{align}
where $\widetilde{\varphi}^{k}(\xi):=\varphi^{x,\,t}_{k}(\xi):=|{\rm det}(M^{-1}_{k}M_{x,t})|\varphi(M^{-1}_{k}M_{x,t}\xi)$. Here and hereafter, for any $\theta(x,t+kJ)=x+M_{x,t+kJ}(\mathbb{B}^n)\in\Theta$, we denote $M_{k}:=M_{x,t+kJ}$.

Furthermore, for any positive integers $L,\,N,\,\widetilde{N}$ and $\widetilde{N}\geq N$, there exist integers $M\geq N+2(n+1)$, $\widetilde{M}\geq \max\{\widetilde{N}+2(n+1)+L(\lfloor1/({a_{6}J})\rfloor+1),\,M\}$ and constant $C$ (depending on $L$, $N$, $\widetilde{N}$ and parameters $p(\Theta)$, but independent of the choice of $\psi$) such that
\begin{align}\label{e4.8}
\|\eta^{k}\|_{\mathcal{S}_{N,\,\widetilde{N}}}\leq C2^{-kL}\|\psi\|_{\mathcal{S}_{M,\,\widetilde{M}}}.
\end{align}
\end{proposition}
\begin{proof}
The proof is divided into 3 steps since it is a little complicated.

\textbf{Step\,1.} Show \eqref{e4.7}
holds pointwise everywhere. By scaling of $\varphi$ and $\mathbb{B}^{n}$ we can assume that $\int_{\rn}\varphi(x)dx=1$ and $|\hat{\varphi}(\xi)|\geq1/2$ for $\xi\in2a_{5}\mathbb{B}^{n}$. Consider a
infinitely differentiable  function $\zeta$ such that $\zeta\equiv1$ on $\mathbb{B}^{n}$ and supp $\zeta \subset 2\mathbb{B}^{n}$. Fix $x\in\mathbb{R}^n$ and $t\in\mathbb{R} $. Define a sequence of functions $\{\zeta_{k}\}_{k=0}^{\fz}$, where $\zeta_{0}:=\zeta$,
$$\zeta_{k}(\xi):=\zeta\lf(M^{T}_{k}(M^{T}_{x,t})^{-1}\,\xi\r)-\zeta\lf(M^{T}_{k-1}(M^{T}_{x,t})^{-1}\xi\r)\ \ \ \  \text{for}\ \ k \geq1,\ \ \xi\in \mathbb{R}^n,$$
where $M^{T}$ denotes the transposed matrix of $M$.
By $\|M\|=\|M^{T}\|$ and \eqref{e2.2}, for any given $\xi\in {\mathbb{R}^{n}}$, we obtain
\begin{align*}
\lf|M^{T}_{k}(M^{T}_{x,t})^{-1}\,\xi \r| \leq \lf\|M^{T}_{k}(M^{T}_{x,t})^{-1} \r\||\xi| \leq a_{5}2^{-a_{6}kJ}|\xi|\rightarrow 0 \ \  \text{as} \ \ k\rightarrow \fz.
\end{align*}
From this, we deduce that, for any $\xi\in {\mathbb{R}^{n}}$, there exists $k$ large enough such that $M^{T}_{k}(M^{T}_{x,t})^{-1}\,\xi \in \mathbb{B}^{n}$. By this, we have
$$\sum_{k=0}^{\fz}\zeta_{k}(\xi)=1,  \ \ \   \xi\in {\mathbb{R}^{n}}.$$
Thus,
$$\hat{\psi}(\xi)=\sum_{k=0}^{\fz}\frac{\zeta_{k}(\xi)}{\hat{\varphi}\lf(M^{T}_{k}(M^{T}_{x,t})^{-1}\,\xi\r)}
\hat{\psi}(\xi){\hat{\varphi}\lf(M^{T}_{k}(M^{T}_{x,t})^{-1}\,\xi\r)}.$$
For $k\in N_0$ define $\eta^{k}$ by
\begin{align*}
\hat{\eta^{k}}(\xi):=\frac{\zeta_{k}(\xi)}{\hat{\varphi}\lf(M^{T}_{k}(M^{T}_{x,t})^{-1}\,\xi\r)}
\hat{\psi}(\xi).
\end{align*}
By ${\hat{\varphi}(M^{T}_{k}(M^{T}_{x,t})^{-1}\,\xi)}=\hat{\widetilde{\varphi}^{k}}(\xi)$ and the choice of ${\{\eta^{k}\}}_{k=0}^{\fz}$
, we know that \eqref{e4.7} holds true pointwise everywhere.

\textbf{Step\,2.} Prove \eqref{e4.7} holds in $\mathcal{S}$. We first claim that for any positive integers $N$, $\widetilde{N}\in\mathbb{N}_{0}$ and $\widetilde{N}\geq N$, there exists a constant $C>0$ so that
\begin{align}\label{e2.20}
 \sup_{x\in{\mathbb{R}^n}}\sup_{t\in\mathbb{R}}\|\eta^{k}*\widetilde{\varphi}^{k}\|_{\mathcal{S}_{N,\widetilde{N}}}
 \leq C\|\varphi\|_{\mathcal{S}_{N+n+1,\widetilde{N}+n+1}}\|\eta^{k}\|_{\mathcal{S}_{N,\widetilde{N}}}.
 \end{align}
 Indeed, for any multi-index $\alpha$, $|\alpha|\leq N$, $k\in\mathbb{N}_0$, and $x,\,\xi\in \mathbb{R}^n$, by \eqref{e2.2}, we have
\begin{align*}
&(1+|\xi|)^{\widetilde{N}}|\partial^{\alpha}(\eta^{k}*\widetilde{\varphi}^{k})(\xi)|=(1+|\xi|)^{\widetilde{N}}|\lf(\partial^{\alpha}\eta^{k}*\widetilde{\varphi}^{k}\r)(\xi)|\\
&\quad\leq\int_{\rn}(1+|\xi-y|)^{\widetilde{N}}\lf|\partial^{\alpha}\eta^{k}(\xi-y)\r|(1+|y|)^{\widetilde{N}}\lf|\widetilde{\varphi}^{k}(y)\r|dy\\
&\quad\leq2^{kJ}\lf\|\eta^{k}\r\|_{\mathcal{S}_{N,\widetilde{N}}}\|\varphi\|_{\mathcal{S}_{N+n+1,\widetilde{N}+n+1}}\\
&\ \ \ \quad\times\int_{\rn}\lf(1+\lf\|M_{x,t}^{-1}M_{k}\r\|\lf|M_{k}^{-1}M_{x,t}y\r|\r)^{\widetilde{N}}\lf(1+\lf|M_{k}^{-1}M_{x,t}y\r|\r)^{-(\widetilde{N}+n+1)}dy\\
&\quad\leq(a_{5})^{\widetilde{N}}2^{kJ}\lf\|\eta^{k}\r\|_{\mathcal{S}_{N,\widetilde{N}}}\|\varphi\|_{\mathcal{S}_{N+n+1,\widetilde{N}+n+1}}\int_{\rn}(1+\lf|M_{k}^{-1}M_{x,t}y\r|)^{-(n+1)}dy\\
&\quad\leq C\|\varphi\|_{\mathcal{S}_{N+n+1,\widetilde{N}+n+1}}\lf\|\eta^{k}\r\|_{\mathcal{S}_{N,\widetilde{N}}},
\end{align*}
which implies \eqref{e2.20} holds true.

Now, by Step\,1, \eqref{e2.20} and \eqref{e4.8}, we may obtain that the convergence of the series  \eqref{e4.7} is in $\mathcal{S}$. Therefore, it remains to show \eqref{e4.8}.

\textbf{ Step\,3.} Show \eqref{e4.8}. Firstly, we claim that, for any $\eta\in\mathcal{S}$, positive integers $N,\,{N'}$ and $ {N'}\geq N$, there exists a constant $C>0$ such that
\begin{align}\label{e4.9}
\|\hat{\eta}\|_{\mathcal{S}_{N,\,{N'}}}\leq C\|\eta\|_{\mathcal{S}_{N+n+1,\,{N'}+n+1}}.
\end{align}
Indeed, for any multi-indices $|\alpha|,\,|\beta|\leq N$, we have
\begin{align*}
(2\pi i)^{|\beta|}\xi^{\beta} \partial^{\alpha}\hat{\eta}(\xi)=(-2\pi i)^{|\alpha|}\int_{\mathbb{R}^{n}} \partial^{\beta}\lf[x^{\alpha}{\eta}(x)\r]e^{-2\pi i\langle x,\xi \rangle}dx.
\end{align*}
Hence, by multiplying and dividing the right hand side by $(1+|x|)^{n+1}$, we have
\begin{align*}
\lf|\xi^{\beta} \partial^{\alpha}\hat{\eta}(\xi)\r|\leq(2\pi)^{|\alpha|-|\beta|}\sup_{x\in \mathbb{R}^{n}}(1+|x|)^{n+1}\lf|\partial^{\beta}\lf[x^{\alpha}{\eta}(x)\r]\r|
\int_{\mathbb{R}^{n}}(1+|x|)^{-n-1}dx,
\end{align*}
which implies \eqref{e4.9}.

Thus, to show \eqref{e4.8}, by \eqref{e4.9}, it suffices to show that there exists a constant $C$ (independent of $\psi\in\mathcal{S}$) such that
\begin{align}\label{e4.10x}
\lf\|\hat{\eta^{k}}\r\|_{\mathcal{S}_{N+n+1,\,\widetilde{N}+n+1}}\leq C2^{-kL}\|\psi\|_{\mathcal{S}_{M,\,\widetilde{M}}}, \ \ \ k\in \mathbb{N}_0.
\end{align}

 We first claim
\begin{align}\label{e4.10}
\sup_{\xi\in\mathbb{R}^{n}}\sup_{|\alpha|\leq N+n+1}\lf|\partial^{\alpha}\lf(\zeta_{k}(\cdot)/\hat{\varphi}\lf(M^{T}_{k}\lf(M^{T}_{x,t}\r)^{-1}\cdot\r)\r)(\xi)\r|\leq C_1,
\end{align}
where $C_1:=C(\zeta,\,N,\,n,\,p(\Theta))$ is a positive constant. Indeed, let $$\tilde{{\zeta}}(\xi):=[\zeta(\xi)-\zeta(M^{T}_{(k-1)}(M^{T}_{x,k})^{-1}\xi)]/\hat{\varphi}(\xi).$$
Obviously,
\begin{align}\label{e2.21}
 \text{supp}\, \tilde{{\zeta}}\subseteq \,\text{supp}\, {{\zeta}} \cup\, \text{supp}\, {{\zeta}}(M^{T}_{k-1}(M^{T}_{k})^{-1}\cdot).
 \end{align}
If $M^{T}_{k-1}(M^{T}_{k})^{-1}\xi \in 2\mathbb{B}^{n}$, then we have
\begin{align}\label{e2.24x}
\xi \in 2M^{T}_{k}(M^{T}_{k-1})^{-1}(\mathbb{B}^{n}).
 \end{align}
Furthermore, by $\|M^{T}\|=\|M\|$ and \eqref{e2.2}, we obtain
$$\lf\|M^{T}_{k}(M^{T}_{k-1})^{-1}\r\|\leq a_{5}2^{-a_{6}J}\leq a_{5},$$
which implies $M^{T}_{k}(M^{T}_{k-1})^{-1}(\mathbb{B}^{n})\subseteq a_{5}\mathbb{B}^{n}$. By this, \eqref{e2.24x}, \eqref{e2.21}, supp ${\zeta} \subset 2\mathbb{B}^{n}$ and $a_{5}\geq1$ (see \eqref{ef2.4}), we have
$$ \text{supp}\, \tilde{\zeta} \subset 2\mathbb{B}^{n}\cup 2a_{5}\mathbb{B}^{n}=2a_{5}\mathbb{B}^{n}.$$
Using this
and $|\hat{\varphi}(\xi)|\geq1/2$ for any $\xi\in2a_{5}\mathbb{B}^{n}$, we obtain
\begin{align}\label{e2.24}
\sup_{\xi\in\mathbb{R}^{n}}\sup_{|\alpha|\leq N+n+1}|\partial^{\alpha}\tilde{\zeta}(\xi)|\leq C,
\end{align}
where $C:=C(\zeta,\,N,\,n,\,p(\Theta))$ is a positive constant.

From $\tilde{{\zeta}}(M^{T}_{k}(M^{T}_{x,t})^{-1}\xi)=\zeta_{k}(\xi)/\hat{\varphi}(M^{T}_{k}(M^{T}_{x,t})^{-1}\xi)$, the chain rule, $\|M^{T}\|=\|M\|$, \eqref{e2.2} and \eqref{e2.24}, it follows that
\begin{align*}
&\sup_{\xi\in\mathbb{R}^{n}}\sup_{|\alpha|\leq N+n+1}\lf|\partial^{\alpha}\lf(\zeta_{k}(\cdot)/\hat{\varphi}\lf(M^{T}_{k}\lf(M^{T}_{x,t}\r)^{-1}\cdot\r)\r)(\xi)\r|\\
&\quad=\sup_{\xi\in\mathbb{R}^{n}}\sup_{|\alpha|\leq N+n+1}\lf|\partial^{\alpha}\lf(\tilde{{\zeta}}\lf(M^{T}_{k}\lf(M^{T}_{x,t}\r)^{-1}\cdot\r)\r)(\xi)\r|\nonumber\\
&\quad\leq C \sup_{\xi\in\mathbb{R}^{n}}\sup_{|\alpha|\leq N+n+1}\lf\|M^{T}_{k}\lf(M^{T}_{x,t}\r)^{-1}\r\|^{|\alpha|}\lf|\lf(\partial^{\alpha}\tilde{{\zeta}}\r)\lf(M^{T}_{k}\lf(M^{T}_{x,t}\r)^{-1}\xi\r)\r|\nonumber\\
&\quad\leq C \sup_{\xi\in\mathbb{R}^{n}}\sup_{|\alpha|\leq N+n+1}\lf(a_{5}2^{-a_{6}kJ}\r)^{|\alpha|}\lf|\lf(\partial^{\alpha}\tilde{{\zeta}}\r)\lf(M^{T}_{k}\lf(M^{T}_{x,t}\r)^{-1}\xi\r)\r|\nonumber\\
& \quad\leq  Ca_{5}^{N+n+1}\sup_{\xi\in\mathbb{R}^{n}}\sup_{|\alpha|\leq N+n+1}\lf|\partial^{\alpha}\tilde{{\zeta}}(\xi)\r|\leq C_{1},
\end{align*}
which means \eqref{e4.10} holds true.

Notice that for any $\xi\in a_{5}^{-1}M^{T}_{x,t}(M^{T}_{k-1})^{-1}(\mathbb{B}^{n})$, by \eqref{e2.2} and $\zeta\equiv1$ on $\mathbb{B}^{n}$, we have $\zeta_{k}(\xi)=0$. By this, the product rule and \eqref{e4.10}, we obtain
\begin{align}\label{e2.26}
\lf\|\hat{\eta^{k}}\r\|_{\mathcal{S}_{N+n+1,\,\widetilde{N}+n+1}}
&= \sup_{\xi\notin a_{5}^{-1} M^{T}_{x,t}(M^{T}_{k-1})^{-1}(\mathbb{B}^{n})}\sup_{|\alpha|\leq N+n+1}\lf\{(1+|\xi|)^{\widetilde{N}+n+1}\r.\\
&\qquad\hs \lf. \times\lf|\partial^{\alpha}\lf[\zeta_{k}(\cdot)\hat{\psi}(\cdot)/\hat{\varphi}\lf(M^{T}_{k}\lf(M^{T}_{x,t}\r)^{-1}\cdot\r)\r](\xi)\r|\r\}\nonumber\\
&\leq C_{N,\,n}C_1\sup_{\xi\notin a_{5}^{-1}M^{T}_{x,t}(M^{T}_{k-1})^{-1}(\mathbb{B}^{n})}\sup_{|\alpha|\leq N+n+1}(1+|\xi|)^{\widetilde{N}+n+1}\lf|\partial^{\alpha}\hat{\psi}(\xi)\r|\nonumber\\
&\leq C_{N,\,n}\,C_{1} \sup_{\xi\notin a_{5}^{-1}M^{T}_{x,t}(M^{T}_{k-1})^{-1}(\mathbb{B}^{n})}(1+|\xi|)^{-L(\lfloor\frac{1}{a_{6}J}\rfloor+1)}\lf\|\hat{\psi}\r\|_{\mathcal{S}_{M-n-1,\,\widetilde{M}-n-1}}\nonumber,
\end{align}
where $M\geq N+2(n+1)$ and $\widetilde{M}\geq \max{\{\widetilde{N}+2(n+1)+L(\lfloor1/({a_{6}J})\rfloor+1),\,M\}}$.
From $\|M^{T}\|=\|M\|$ and \eqref{e2.2}, we deduce that
$$\|M^{T}_{k-1}(M^{T}_{x,\,t})^{-1}\| \leq a_{5}2^{-a_{6}(k-1)J}$$
and hence
$$M^{T}_{k-1}(M^{T}_{x,\,t})^{-1}(\mathbb{B}^{n})\subseteq a_{5}2^{-a_{6}(k-1)J}\mathbb{B}^{n},$$
which implies
$$  \frac{2^{a_{6}(k-1)J}}{a_{5}}\mathbb{B}^{n}\subseteq M^{T}_{x,\,t}(M^{T}_{k-1})^{-1}(\mathbb{B}^{n}).$$
By this and $\xi\notin a_{5}^{-1}M^{T}_{x,\,t}(M^{T}_{k-1})^{-1}(\mathbb{B}^{n})$, we have
\begin{align}\label{e2.27}
|\xi|\geq  2^{a_{6}(k-1)J}/(a_{5})^{2}.
  \end{align}
Consequently, inserting \eqref{e2.27} into \eqref{e2.26} and
by \eqref{e4.9}, we have
 \begin{align*}
 \lf\|\hat{\eta^{k}}\r\|_{\mathcal{S}_{N+n+1,\,\widetilde{N}+n+1}}
\leq C2^{-kL}\lf\|\hat{\psi}\r\|_{\mathcal{S}_{M-n-1,\,\widetilde{M}-n-1}}
\leq C2^{-kL}\|\psi\|_{\mathcal{S}_{M,\,\widetilde{M}}}
\end{align*}
and hence \eqref{e4.10x} holds true. This finishes the proof of Proposition \ref{l4.5}.
\end{proof}



\section{Maximal Function Characterizations of $H^p(\Theta)$}\label{s3}
In this section, we show the maximal function characterizations of $H^p(\Theta)$ using
the radial, the non-tangential and the tangential maximal function of a single test function $\varphi\in \mathcal{S}$.

\begin{theorem}\label{t4.1}
Let $\Theta$ be an ellipsoid cover, $0<p<\fz$ and $\varphi\in \cs$ satisfy $\int_{\mathbb{R}^{n}}\varphi(x)\,dx\neq0$. Then for any $f\in\mathcal{S'}$, the following are mutually equivalent:
\begin{align}\label{e3.1}
f\in {H^p(\Theta)};
\end{align}
\begin{align}\label{e3.2}
M_\varphi f\in{L^p};
\end{align}
\begin{align}\label{e3.3}
M^0_\varphi f\in{L^p};
\end{align}
\begin{align}\label{e3.4}
T^N_\varphi f\in{L^p}&, \ \ \ N>\frac{1}{a_{6}p}.
\end{align}
In this case,
\begin{align*}
\|f\|_{H^p(\Theta)}=\lf\|M^{0}f\r\|_{L^{p}}\leq C_{1}\lf\|T^N_\varphi f\r\|_{L^{p}}\leq C_{2}\lf\|M_{\varphi} f\r\|_{L^{p}} \leq C_{3}\lf\|M^0_{\varphi} f\r\|_{L^{p}}\leq C_{4}\|f\|_{H^p(\Theta)},
\end{align*}
where the positive constants $ C_{1}$, $C_{2}$, $C_{3}$ and $C_{4}$ are independent of $f$.
\end{theorem}

The framework to prove Theorem \ref{t4.1} is motivated by Fefferman and Stein \cite{fe}, \cite[Chapter III]{s}, and Bownik \cite[p.42, Theorem 7.1]{b}.

Inspired by Fefferman and Stein \cite[p.\,97]{s} and Bownik \cite[p.\,47]{b}, we now start with maximal functions obtained from truncation with an additional extra decay term. Namely, for $t_{0}<0$ representing the truncation level
and real number $L\geq0$ representing the decay level, we define the {\it radial}, the {\it non-tangential}, the {\it tangential}, the {\it grand radial}, and the {\it grand non-tangential maximal functions}, respectively as
\begin{align*}
&M_{\varphi}^{0\,(t_{0},\,L)}f(x):=\sup_{t\geq t_0}|(f\ast\varphi_{x,\,t})(x)|\,\lf(1+\lf|M_{x,\,t_{0}}^{-1}\,x\r|\r)^{-L}\lf(1+{2^{t+{t_{0}}}}\r)^{-L},\\
&M_{\varphi}^{(t_{0},\,L)}f(x):=\sup_{t\geq t_0}\sup_{y\in\theta(x,\,t)}|(f\ast\varphi_{x,\,t})(y)|\,\lf(1+\lf|M_{x,\,t_{0}}^{-1}\,y\r|\r)^{-L}\lf(1+{2^{t+{t_{0}}}}\r)^{-L},\\
&T_{\varphi}^{N\,(t_{0},\,L)}f(x):=\sup_{t\geq t_0}\sup_{y\in\mathbb{R}^{n}}\frac{|(f\ast\varphi_{x,\,t})(y)|}{\lf[1+\lf|M_{x,\,t}^{-1}\,(x-y)\r|\r]^{N}}
\frac{1}{\lf(1+{2^{t+{t_{0}}}}\r)^{L}\lf(1+\lf|M_{x,\,t_{0}}^{-1}\,y\r|\r)^{L}},\\
&M_{N,\,\wz{N}}^{0\,(t_{0},\,L)}f(x):=\sup_{\varphi \in\mathcal{S}_{N,\,\wz{N}}}M_{\varphi}^{0\,(t_{0},\,L)}f(x)\\
\text{and}\\
&M_{N,\,\wz{N}}^{(t_{0},\,L)}f(x):=\sup_{\varphi \in\mathcal{S}_{N,\,\wz{N}}}M_{\varphi}^{(t_{0},\,L)}f(x).
\end{align*}

The following Lemma \ref{l4.6} guarantees the control of the tangential by the non-tangential maximal function in $L^{p}(\mathbb{B}_{K})$ independent of $t_{0}$ and $L$.

\begin{lemma}\label{l4.6}
Let $K>0$, $\mathbb{B}_{K}=\{y\in\rn:|y|<K\}$ and $\Theta$ an ellipsoid cover which is pointwise continuous. Suppose $p>0,\,N>1/(a_{6}\,p)$ and $\varphi\in \mathcal{S}$. Then there exists a positive constant $C$ such that for any $t_{0}<0,\,L\geq0$ and $f\in \mathcal{S'}$,
$$\lf\|T_{\varphi}^{N\,(t_{0},\,L)}f\r\|_{L^{p}(\mathbb{B}_{K})}\leq C\lf\|M_{\varphi}^{(t_{0},\,L)}f\r\|_{L^{p}(\mathbb{B}_{K})}.$$
\end{lemma}
\begin{proof}
Let $x$ be a fixed point in $\mathbb{B}_K$. Consider function $F_x:\mathbb{R}^{n} \times\mathbb{R}\longrightarrow [0,\,\fz)$ given by
$$F_x(y,\,t):=|(f\ast\varphi_{x,\,t})(y)|^{p}\,\lf(1+\lf|M_{x,\,t_{0}}^{-1}\,y\r|\r)^{-pL}(1+{2^{t+{t_{0}}}})^{-pL}.$$
 By \cite[(3.8)]{bll} and the pointwise continuity of $\Theta$, we know that $F_x(y,\,t)$ is
continuous for variable $x$ on $\rn$. Let $F_{l}^{* \,t_{0}}$ be as in \eqref{e3.3x} with $l=0$. When $y\in \theta(x,\,t)$, we have $M_{x,\,t}^{-1}\,(x-y)\in \mathbb{B}^{n}$ and hence $|M_{x,\,t}^{-1}\,(x-y)|< 1$. If $t\geq t_0$, then
$$F_x(y,\,t)\lf[1+\lf|M_{x,\,t}^{-1}\,(x-y)\r|\r]^{-pN}\leq F_{0}^{\ast\, t_{0}}(x).$$
 When $y\in \theta(x,\,t-kJ)\backslash \theta(x,\,t-(k-1)J)$ for some $k\geq1$, we have
 \begin{align}\label{e3.13x}
M_{x,\,t}^{-1}\,(x-y) \notin M_{x,\,t}^{-1}\,M_{x,\,t-(k-1)J}\,(\mathbb{B}^{n}).
\end{align}
By \eqref{e2.2}, we obtain
$$\lf\|M_{x,\,t-(k-1)J}^{-1}\, M_{x,\,t}\r\| \leq {a_5}{2^{-a_{6}(k-1)J}}$$
and hence
$$M_{x,\,t-(k-1)J}^{-1}\, M_{x,\,t}(\mathbb{B}^{n})\subseteq{a_5}{2^{-a_{6}(k-1)J}}\mathbb{B}^{n} ,$$
which implies
$$({2^{a_{6}(k-1)J}}/{a_5})\mathbb{B}^{n}\subseteq M_{x,\,t}^{-1}\,M_{x,\,t-(k-1)J}(\mathbb{B}^{n}).$$
From this and \eqref{e3.13x}, it follows that $|M_{x,\,t}^{-1}\,(x-y)|\geq {2^{a_{6}(k-1)J}}/{a_5}$. Thus, for any $t\geq t_0$, we have
$$F_x(y,\,t)\lf[1+\lf|M_{x,\,t}^{-1}\,(x-y)\r|\r]^{-pN}\leq{a_5}^{pN}{2^{-pN a_{6}(k-1)J}} F_{k}^{\ast\, t_{0}}(x).$$
By taking supremum over all $y\in \mathbb{R}^n$ and $t\geq t_{0}$, we know that
\begin{align*}
\lf[T_{\varphi}^{N\,(t_{0},\,L)}f(x)\r]^{p}\leq {a_5}^{pN}\sum_{k=0}^{\fz}{2^{-pN a_{6}(k-1)J}} F_{k}^{\ast\, t_{0}}(x).
\end{align*}
Therefore, using this and Proposition \ref{l4.4}, we obtain
\begin{align*}
\lf\|T_{\varphi}^{N\,(t_{0},\,L)}f\r\|_{L^{p}(\mathbb{B}_{K})}^{p}&\leq {a_5}^{pN}\sum_{k=0}^{\fz}{2^{-pN a_{6}(k-1)J}} \int_{\mathbb{B}_{K}}F_{k}^{\ast\, t_{0}}(x)dx\\
&\leq C{a_5}^{pN}\sum_{k=0}^{\fz}{2^{-pN a_{6}(k-1)J}} 2^{kJ}\int_{\mathbb{B}_{K}}F_{0}^{\ast\, t_{0}}(x)dx\\
&= C{'}\lf\|M_{\varphi}^{(t_{0},\,L)}f\r\|_{L^{p}(\mathbb{B}_{K})}^{p},
\end{align*}
where $C{'}:=C{a_5}^{pN}2^{pN a_{6}J}\sum_{k=0}^{\fz}{2^{(1-pN a_{6})kJ}}=C {a_5}^{pN} 2^{J}/(1-2^{(1-pN a_{6})J})$.
\end{proof}

The following Lemma \ref{l4.7} gives the pointwise majorization of the grand radial maximal function by the tangential one, which
is a variable anisotropic extension of \cite[Lemma 7.5]{b}.
\begin{lemma}\label{l4.7}
Suppose $\varphi\in \mathcal{S}$ and $\int_{\mathbb{R}^{n}}\varphi(x)\,dx\neq0$. For given positive integers N and L, there exist integers $M>0$, $\widetilde{M}\geq M$ and constant $C>0$ such that, for any $f\in \mathcal{S'}$ and $t_{0}<0$, it holds true that
$$M_{M,\,\widetilde{M}}^{0\,(t_{0},\,L)}f(x)\leq C T_{\varphi}^{N\,(t_{0},\,L)}f(x),\ \ \ x\in \mathbb{R}^{n}.$$
\end{lemma}
\begin{proof}
For any $\psi\in\mathcal{S}$, by Proposition \ref{l4.5}, there exists a sequence of test functions $\{\eta^{k}\}_{k=0}^{\fz}$ such that \eqref{e4.7} and \eqref{e4.8} hold true. Thus, for fixed $x\in \mathbb{R}^{n}$ and $t\geq t_{0}$, we have
\begin{align*}
|(f*\psi_{x,\,t})(x)|&=\lf|\lf[f*\sum_{k=0}^{\fz}\lf(\eta^{k}*\widetilde{\varphi}^{k}\r)_{x,\,t}\r](x)\r|\\
&\leq C\lf|\lf[f*\sum_{k=0}^{\fz}\lf|{\rm det}\lf(M^{-1}_{k}\r)\r|\int_{\mathbb{R}^{n}}\eta^{k}(y)\,\varphi\lf(M^{-1}_{k}(\cdot-M_{x,t}y)\r)dy\r](x)\r|\\
&=C\lf|\lf[f*\sum_{k=0}^{\fz}\lf|\text{det}\lf(M^{-1}_{k} M_{x,t}^{-1}\r)\r|\int_{\mathbb{R}^{n}}\eta^{k}\lf(M^{-1}_{x,t}y\r)
\varphi\lf(M^{-1}_{k}(\cdot-y)\r)dy\r](x)\r|\\
&\leq C\sum_{k=0}^{\fz}\lf|\lf[f*\lf(\eta^{k}\r)_{x,\,t}*{\varphi}_{x,\,t+kJ}\r](x)\r|\\
&\leq C\sum_{k=0}^{\fz}\int_{\mathbb{R}^{n}}\lf|f*{\varphi}_{x,\,t+kJ}(x-y)\r|\lf|\lf(\eta^{k}\r)_{x,\,t}(y)\r|dy\\
&\leq C T_{\varphi}^{N\,(t_{0},\,L)}f(x)\sum_{k=0}^{\fz}\int_{\mathbb{R}^{n}}\lf(1+\lf|M_{k}^{-1}\,y\r|\r)^{N}\\
&\ \ \ \times\lf(1+\lf|M_{x,\,t_{0}}^{-1}\,(x-y)\r|\r)^{L}\lf(1+2^{t+t_{0}+kJ}\r)^{L}\lf|\lf(\eta^{k}\r)_{x,\,t}(y)\r|dy.\\
\end{align*}
Therefore,
\begin{align}\label{e3.6}
M_{\psi}^{0\,(t_{0},\,L)}f(x)&\leq T_{\varphi}^{N\,(t_{0},\,L)}f(x)\,\sup_{t\geq t_{0}}\,\sum_{k=0}^{\fz}\int_{\mathbb{R}^{n}}\lf(1+\lf|M_{k}^{-1}\,y\r|\r)^{N}\\
&\ \ \ \times\frac{\lf(1+\lf|M_{x,\,t_{0}}^{-1}\,(x-y)\r|\r)^{L}\lf(1+2^{t+t_{0}+kJ}\r)^{L}}{\lf(1+\lf|M_{x,\,t_{0}}^{-1}\,x\r|\r)^{L}\lf(1+2^{t+t_{0}}\r)^{L}}\lf|\lf(\eta^{k}\r)_{x,\,t}(y)\r|dy\nonumber\\
&=:T_{\varphi}^{N\,(t_{0},\,L)}f(x){\rm I}_{t_0}(x)\nonumber.
\end{align}

To estimate ${\rm I}_{t_0}(x)$, by
\begin{align*}
\frac{1+2^{t+t_{0}+kJ}}{1+2^{t+t_{0}}}=\frac{2^{kJ}(2^{-kJ}+2^{t+t_{0}})}{1+2^{t+t_{0}}}\leq C2^{kJ}
\end{align*}
and
\begin{align}\label{e4.11}
1+|x+y|\leq 1+|x|+|y|
\leq(1+|x|)(1+|y|),\ \ x,\,y\in\mathbb{R}^{n},
\end{align}
 we obtain
\begin{align*}
{\rm I}_{t_0}(x)&\leq C\sup_{t\geq t_{0}}\sum_{k=0}^{\fz}2^{t+kJL}\int_{\mathbb{R}^{n}}\lf(1+\lf|M_{k}^{-1}\,y\r|\r)^{N}
{\lf(1+\lf|M_{x,\,t_{0}}^{-1}\,y\r|\r)^{L}}\lf|\eta^{k}\lf(M_{x,\,t}^{-1}y\r)\r|dy\\
&\leq C\sup_{t\geq t_{0}}\sum_{k=0}^{\fz}2^{kJL}\int_{\mathbb{R}^{n}}\lf(1+\lf\|M_{k}^{-1}M_{x,\,t}\r\|\,|y|\r)^{N}
{\lf(1+\lf\|M_{x,\,t_{0}}^{-1}M_{x,\,t}\r\|\,|y|\r)^{L}}\lf|\eta^{k}(y)\r|dy,
\end{align*}
which, together with
$$\|M_{k}^{-1}M_{x,\,t}\|\leq a_{3}2^{a_{4}kJ}\ \ \text{and} \ \ \|M_{x,\,t_{0}}^{-1}M_{x,\,t}\|\leq a_{5}2^{-a_{6}(t-t_{0})}\leq a_{5}\ \text{(by } \ t\ge t_0 \ \text{and} \ \eqref{e2.2}),$$
further implies that
\begin{align}\label{e3.8}
{\rm I}_{t_0}(x)
&\leq C\sum_{k=0}^{\fz}2^{kJ(L+a_{4}N)}\int_{\mathbb{R}^{n}}\lf(1+|y|\r)^{N+L}
\lf|\eta^{k}(y)\r|dy\\
&\leq C\sum_{k=0}^{\fz}2^{kJ(L+a_{4}N)}
\lf\|\eta^{k}\r\|_{\mathcal{S}_{N+n+L,\,\widetilde{N}+n+L}}\nonumber.
\end{align}
By Proposition \ref{l4.5}, there exist positive integers $M$ and $\widetilde{M}$ such that
\begin{align}\label{e3.9}
\lf\|\eta^{k}\r\|_{\mathcal{S}_{N+n+L,\,\widetilde{N}+n+L}}\leq C2^{-kJ[L+(\lfloor a_{4}\rfloor+1) N]}\|\psi\|_{\mathcal{S}_{M,\,\,\widetilde{M}}}.
\end{align}
Thus, combining with \eqref{e3.6},  \eqref{e3.8} and  \eqref{e3.9}, we finally obtain
\begin{align*}
M_{M,\widetilde{M}}^{0\,(t_{0},\,L)}f(x)&=\sup_{\psi\in{\mathcal{S}_{M,\widetilde{M}}}}M_{\psi}^{0\,(t_{0},\,L)}f(x)\\
&\leq C \sum_{k=0}^{\fz}2^{kJN[a_{4}-(\lfloor a_{4}\rfloor +1)]}T_{\varphi}^{N\,(t_{0},\,L)}f(x)=C\,T_{\varphi}^{N\,(t_{0},\,L)}f(x).
\end{align*}
This finishes the proof of Lemma \ref{l4.7}.
\end{proof}

The following Lemma \ref{l4.9} shows that the radial and the grand non-tangential maximal functions are pointwise equivalent, which
is a variable anisotropic extension of \cite[Proposition 3.10]{b}.
\begin{lemma}{\rm{\cite[Theorem 3.4]{bll}}}\label{l4.9}
For any $N,\,\wz{N}\in\mathbb{N}$ with $N\le\wz{N}$, there exists a positive constant $C:=C(\wz{N})$ such that for any $f\in\mathcal{S'}$,
$$M^0_{N,\wz{N}}f(x)\le M_{N,\wz{N}} f(x)\le C M^0_{N,\wz{N}}f(x), \ \ \ \  x\in\rn.$$
\end{lemma}

The following Lemma \ref{l4.8} is a variable anisotropic extension of \cite[p.\,46, Lemma 7.6]{b}.

\begin{lemma}\label{l4.8}
Let $\varphi\in \mathcal{S}$, $f\in \mathcal{S'}$ and $K\in (0,\infty)$. Then for every $M>0$ and $t_{0}<0$ there exist $L>0$ and $N'>0$ large enough such that
\begin{align}\label{e4.12}
M_{\varphi}^{(t_{0},\,L)}f(x)\leq C 2^{-t_{0}(2a_{4}{N'}+2L+a_{4}L)}(1+|x|)^{-M},\ \ \ x\in \mathbb{B}_{K}=\{y\in\rn:|y|<K\},
\end{align}
where $C$ is a positive constant dependent on $p(\Theta)$, $N'$, $f$, $\varphi$ and $K$.
\end{lemma}
\begin{proof}
For any $\vz\in\mathcal{S}$, there exist an integer $N>0$ and positive constant $C:=C(\varphi)$ such that, for any ${N'}\geq N$
and $y\in\rn$,
\begin{align}\label{ea4.13}
|f*\varphi(y)|\leq C\|\varphi\|_{\mathcal{S}_{N,\,{N'}}}(1+|y|)^{{N'}}.
\end{align}
Therefore, for any $t_{0}<0$, $t\geq t_{0}$ and $x\in \mathbb{B}_K$, by \eqref{ea4.13}, we have
\begin{align}\label{e3.27}
&|(f\ast\varphi_{x,\,t})(y)|\,\lf(1+\lf|M_{x,\,t_{0}}^{-1}\,y\r|\r)^{-L}\lf(1+{2^{t+{t_{0}}}}\r)^{-L}\\
&\quad\leq C {2^{{-L}({t+{t_{0}}})}}\|\varphi_{x,\,t}\|_{\mathcal{S}_{N,\,{N'}}}(1+|y|)^{{N'}}\lf(1+\lf|M_{x,\,t_{0}}^{-1}\,y\r|\r)^{-L}\nonumber.
\end{align}
Let us first estimate $\|\varphi_{x,\,t}\|_{\mathcal{S}_{N,\,{N'}}}$. By the chain rule and \eqref{e2.1}, we have
\begin{align}\label{ef3.24}
\|\varphi_{x,\,t}\|_{\mathcal{S}_{N,\,{N'}}}&=|\text{det}M_{x,\,t}^{-1}|\sup_{z\in \mathbb{R}^{n}}\sup_{|\alpha|\leq N}(1+|z|)^{{N'}}\lf|\partial^{\alpha}\lf(\varphi\lf(M_{x,\,t}^{-1}\cdot\r)\r)(z)\r|\nonumber\\
&\leq C2^{t}\sup_{z\in \mathbb{R}^{n}}\sup_{|\alpha|\leq N}(1+|z|)^{{N'}}\lf\|M_{x,\,t}^{-1}\r\|^{|\alpha|}\lf|\lf(\partial^{\alpha}\varphi\r)\lf(M_{x,\,t}^{-1}z\r)\r|\nonumber\\
&\leq C2^{t}\sup_{z\in \mathbb{R}^{n}}\sup_{|\alpha|\leq N}(1+|M_{x,\,t}z|)^{{N'}}\lf\|M_{x,\,t}^{-1}\r\|^{|\alpha|}\lf|\partial^{\alpha}\varphi(z)\r|.
\end{align}

 Notice that for any given $K\in (0,\infty)$, there exists $t_{K}\in \mathbb{R}$ such
that $\mathbb{B}_{K}\subset \theta(0,t_{K})$. Here we might assume $t_{K}<0$ as well. Then, for any  $x\in\mathbb{B}_{K}$, we get $\theta(x,0)\cap\theta(0,t_{K})\neq\emptyset$. Thus, by \eqref{e2.2}, we have
\begin{align}\label{e3.1x}
\lf\|M_{x,\,0}\r\|&=\lf\|M_{0,\,t_{K}}M_{0,\,t_{K}}^{-1}M_{x,\,0}\r\|
\leq\lf\|M_{0,\,t_{K}}\r\|\lf\|M_{0,\,t_{K}}^{-1}M_{x,\,0}\r\|\nonumber\\
&\leq a_{5}2^{a_{6}t_{K}} \lf\|M_{0,\,t_{K}}\r\|:=C_1.
\end{align}
Similarly, we also have
\begin{align}\label{ef3.1x}
\lf\|M_{x,\,0}^{-1}\r\|\leq C_{2}.
\end{align}
Here, $C_1$ and $C_{2}$ are positive constants depending on $K$ and $p(\Theta)$.

Now, let's further estimate \eqref{ef3.24} in the following two cases.

\textbf{Case\,1:}\ $t\geq  0$. By \eqref{e2.2}, \eqref{e3.1x} and \eqref{ef3.1x}, we have
\begin{align*}
\lf\|M_{x,\,t}^{-1}\r\|&=\lf\|M_{x,\,t}^{-1}M_{x,\,0}M_{x,\,0}^{-1}\r\|
\leq\lf\|M_{x,\,t}^{-1}M_{x,\,0}\r\|\lf\|M_{x,\,0}^{-1}\r\|
\leq \lf\|M_{x,\,0}^{-1}\r\|a_{3}^{-1}2^{a_{4}t}=C2^{a_{4}t}
\end{align*}
and
\begin{align*}
|M_{x,\,t}z|&=\lf|M_{x,\,0}M_{x,\,0}^{-1}M_{x,\,t}z\r|\leq\lf\|M_{x,\,0}\r\|\lf|M_{x,\,0}^{-1}M_{x,\,t}z\r|
\leq\lf\|M_{x,\,0}\r\|\lf\|M_{x,\,0}^{-1}M_{x,\,t}\r\||z|\\
&\leq\|M_{x,\,0}\| a_{5}2^{-a_{6}t}|z|\leq C |z|.
\end{align*}
Inserting the above two estimates into \eqref{ef3.24} with $t\geq 0$, we know that
\begin{align}\label{ef3.14}
\|\varphi_{x,\,t}\|_{\mathcal{S}_{N,\,{N'}}}&\leq C2^{t}\sup_{z\in \mathbb{R}^{n}}\sup_{|\alpha|\leq N}(1+|M_{x,\,t}z|)^{{N'}}\lf\|M_{x,\,t}^{-1}\r\|^{|\alpha|}\lf|\partial^{\alpha}\varphi(z)\r|\\
&\leq C2^{t}2^{a_{4}tN}\|\varphi\|_{\mathcal{S}_{N,\,{N'}}}\nonumber.
\end{align}

\textbf{Case\,2:}\ $t_{0}\leq t<0$. By \eqref{e2.2}, \eqref{e3.1x} and \eqref{ef3.1x}, we have
\begin{align*}
\lf\|M_{x,\,t}^{-1}\r\|&=\lf\|M_{x,\,t}^{-1}M_{x,\,0}M_{x,\,0}^{-1}\r\|
\leq\lf\|M_{x,\,t}^{-1}M_{x,\,0}\r\|\lf\|M_{x,\,0}^{-1}\r\|
\leq \lf\|M_{x,\,0}^{-1}\r\|a_{5}2^{a_{6}t}\leq C
\end{align*}
and
\begin{align*}
|M_{x,\,t}z|&=\lf|M_{x,\,0}M^{-1}_{x,\,0}M_{x,\,t}z\r|\leq\|M_{x,\,0}\|\lf|M^{-1}_{x,\,0}M_{x,\,t}z\r|\leq\|M_{x,\,0}\|\lf\|M^{-1}_{x,\,0}M_{x,\,t}\r\||z|\\
&\leq \|M_{x,\,0}\|a_{3}^{-1}2^{-a_{4}t}|z|= C2^{-a_{4}t_{0}}|z|.
\end{align*}
Inserting the above two estimates into \eqref{ef3.24} with $t_0\leq t<0$, we know that
\begin{align}\label{ef3.15}
\|\varphi_{x,\,t}\|_{\mathcal{S}_{N,\,{N'}}}&\leq C2^{t}\sup_{z\in \mathbb{R}^{n}}\sup_{|\alpha|\leq N}(1+|M_{x,\,t}z|)^{{N'}}\lf\|M_{x,\,t}^{-1}\r\|^{|\alpha|}\lf|\partial^{\alpha}\varphi(z)\r|\\
&\leq C2^{-a_{4}t_{0}{N'}}\|\varphi\|_{\mathcal{S}_{N,\,{N'}}}\nonumber.
\end{align}

For any $M>0$, let $L:=M+N'$. For any $t_{0}<0$, $t\geq t_{0}$ and taking some integer $N'>0$ large enough, by \eqref{ef3.14} and \eqref{ef3.15}, we obtain
\begin{align}\label{e3.21x}
2^{{-L}({t+{t_{0}}})}\|\varphi_{x,\,t}\|_{\mathcal{S}_{N,\,{N'}}}\leq C2^{-t_{0}(a_{4}{N'}+2L)}\|\varphi\|_{\mathcal{S}_{N,\,{N'}}}.
\end{align}
Inserting \eqref{e3.21x} into \eqref{e3.27}, we further obtain
\begin{align}\label{e3.2x7}
&|(f\ast\varphi_{x,\,t})(y)|\,\lf(1+\lf|M_{x,\,t_{0}}^{-1}\,y\r|\r)^{-L}\lf(1+{2^{t+{t_{0}}}}\r)^{-L}\\
&\quad\leq C2^{-t_{0}(a_{4}{N'}+2L)}\|\varphi\|_{\mathcal{S}_{N,\,{N'}}}(1+|y|)^{{N'}}\lf(1+\lf|M_{x,\,t_{0}}^{-1}\,y\r|\r)^{-L}\nonumber.
\end{align}
For any $y\in \theta(x,\,t)$, there exists $z\in\mathbb{B}^{n}$ such that $y=x+ M_{x,\,t}z$. By \eqref{e4.11}, we have
\begin{align}\label{e3.28x}
1+|y|=1+|x+ M_{x,\,t}z|\leq (1+|x|)(1+| M_{x,\,t}z|).
\end{align}
If $t\geq  0$, by \eqref{e2.2} and \eqref{e3.1x}, then
\begin{align*}
|M_{x,\,t}z|&=\lf|M_{x,\,0}M_{x,\,0}^{-1}M_{x,\,t}z\r|\leq\lf\|M_{x,\,0}\r\|\lf|M_{x,\,0}^{-1}M_{x,\,t}z\r|
\leq\lf\|M_{x,\,0}\r\|\lf\|M_{x,\,0}^{-1}M_{x,\,t}\r\||z|\\
&\leq\|M_{x,\,0}\| a_{5}2^{-a_{6}t}|z|\leq C.
\end{align*}
If $t_{0}\leq t<0$, by \eqref{e2.2} and \eqref{e3.1x}, then
\begin{align*}
|M_{x,\,t}z|&=\lf|M_{x,\,0}M^{-1}_{x,\,0}M_{x,\,t}z\r|\leq\|M_{x,\,0}\|\lf|M^{-1}_{x,\,0}M_{x,\,t}z\r|\leq\|M_{x,\,0}\|\lf\|M^{-1}_{x,\,0}M_{x,\,t}\r\||z|\\
&\leq \|M_{x,\,0}\|a_{3}^{-1}2^{-a_{4}t}|z|= C2^{-a_{4}t_{0}}.
\end{align*}
Therefore, for any $t\geq t_{0}$, by using the above two estimates,  we have
\begin{align*}
|M_{x,\,t}z|\leq C2^{-a_{4}t_{0}}.
\end{align*}
From this and \eqref{e3.28x}, it follows that
\begin{align}\label{e3.22x}
(1+|y|) \leq C 2^{-a_{4}t_{0}}(1+|x|).
\end{align}
Besides, for any $ t_{0}<0$, by \eqref{e2.2} and \eqref{e3.1x}, we have
\begin{align*}
1+|x|\leq1+\|M_{x,\,0}\|\lf\|M^{-1}_{x,\,0}M_{x,\,t_{0}}\r\|\lf|M^{-1}_{x,\,t_{0}}x\r|
\leq C2^{-a_{4}t_{0}}\lf(1+\lf|M^{-1}_{x,\,t_{0}}x\r|\r).
\end{align*}
Furthermore, for any $y\in \theta(x,\,t)$, we have $x\in M_{x,\,t}(\mathbb{B}^{n})+y$. Thus, there exists  $z\in\mathbb{B}^{n}$ such that $x= M_{x,\,t}z+y$.
Hence, for any $t\geq t_{0}$, by \eqref{e4.11} and \eqref{e2.2}, we obtain
\begin{align*}
\lf(1+\lf|M^{-1}_{x,\,t_{0}}x\r|\r)&=\lf(1+\lf|M^{-1}_{x,\,t_{0}}(y+ M_{x,\,t}z)\r|\r)
\leq\lf(1+\lf|M^{-1}_{x,\,,t_{0}}y\r|\r)\lf(1+\lf\|M^{-1}_{x,\,t_{0}}M_{x,\,t}\r\||z|\r)\\
&\leq\lf(1+\lf|M^{-1}_{x,\,,t_{0}}y\r|\r)\lf(1+a_{5}2^{-a_{6}(t-t_{0})}|z|\r)\leq C \lf(1+\lf|M^{-1}_{x,\,t_{0}}y\r|\r).
\end{align*}
Combining with the above two inequalities, we have
\begin{align}\label{e3.23x}
(1+|M^{-1}_{x,\,t_{0}}y|)\geq C2^{a_{4}t_{0}} (1+|x|).
\end{align}
 Thus, for any $t\ge t_0$ and $y\in \theta(x,t)$, inserting \eqref{e3.22x} and \eqref{e3.23x} into \eqref{e3.2x7} with $L=M+N'$, we obtain
\begin{align*}
&|(f\ast\varphi_{x,\,t})(y)|\,\lf(1+\lf|M_{x,\,t_{0}}^{-1}\,y\r|\r)^{-L}\lf(1+{2^{t+{t_{0}}}}\r)^{-L}
\leq C 2^{-t_{0}(2a_{4}{N'}+2L+a_{4}L)}(1+|x|)^{-M} ,
\end{align*}
which implies \eqref{e4.12} holds true and hence completes the proof of Lemma \ref{l4.8}.
\end{proof}

Note that the above arguement gives the same estimate for the truncated grand maximal function $M_{N,\,\wz{N}}^{0\,(t_{0},\,L)}f(x)$. As a consequence of Lemma
\ref{l4.8}, we can get that for any choice of $t_{0}< 0$ and any $f\in{\mathcal{S'}}$, we can find an appropriate $L>0$ so that the maximal function, say $M_{\varphi}^{(t_{0},\,L)}f$, is bounded and belongs to $L^{p}(\mathbb{B}_{K})$. This  becomes crucial in the proof of Theorem \ref{t4.1}, where we work with truncated maximal functions, The complexity of the preceding argument stems from the fact that a priori we do not know wether $M_{\varphi}^{0}f\in L^{p}$
implies $M_{\varphi}f\in L^{p}$, Instead we must work with variants of maximal functions for which this is satisfied.

\begin{proof}[Proof of Theorem \ref{t4.1}]
Let $\varphi\in \mathcal{S}$ satisfy $\int_{\mathbb{R}^{n}}\varphi(x)\,dx\neq0$. From Remark \ref{r3.5} and the definition of the grand radial maximal function, it follows that
$$\eqref{e3.4} \Rightarrow \eqref{e3.2}\Rightarrow \eqref{e3.3}$$
 and
 $$\eqref{e3.1}\Rightarrow \eqref{e3.3}.$$
 By Lemma \ref{l4.6} applied for $L=0$,
we have
$$\lf\|T_{\varphi}^{N(t_{0},\,0)}f\r\|_{L^{p}} \leq C\, \lf\|M_{\varphi}^{(t_{0},\,0)}f\r\|_{L^{p}} \ \ \ \text{for any} \ \ f\in{\mathcal{S'}}\ \text{and} \ \  t_{0}<0.$$
As $t_{0}\rightarrow {-\fz}$, by the monotone convergence theorem, we obtain
 $$\lf\|T_{\varphi}^{N}f\r\|_{L^{p}}\leq C\lf\|M_{\varphi}f\r\|_{L^{p}},$$
 which shows $\eqref{e3.2}\Rightarrow \eqref{e3.4}$.

Combining Lemma \ref{l4.7} applied for $N>1/(a_{6}\,p)$ and $L=0$ and Lemma \ref{l4.6} applied for $L=0$,
we conclude that there exist integers $M>0$, $\widetilde{M}\geq M$ large enough and positive constant $C$ such that
$$\lf\|M_{M,\,\widetilde{M}}^{0(t_{0},\,0)}f\r\|_{L^{p}}\leq C\lf\|M_{\varphi}^{(t_{0},\,0)}f\r\|_{L^{p}} \ \ \ \text{for any} \ \ f\in{\mathcal{S'}} \ \ \text{and}\ \ t_{0}< 0.$$
As $t_{0}\rightarrow {-\fz}$, by the monotone convergence theorem, we obtain $$\lf\|M_{M,\,\widetilde{M}\,}^{0}f\r\|_{L^{p}}\leq C\lf\|M_{\varphi}f\r\|_{L^{p}}.$$ From this and Proposition \ref{p2.7}, we  deduce that
$$\|f\|_{H^p(\Theta)}=\lf\|M^{0}_{N_p,\,\wz{N}_p} f\r\|_{L^{p}}\leq C \lf\|M_{M,\,\widetilde{M}\,}^{0}f\r\|_{L^{p}}\leq C\lf\|M_{\varphi}f\r\|_{L^{p}}$$
and hence
$\eqref{e3.2}\Rightarrow \eqref{e3.1}.$ It remains to show $\eqref{e3.3}\Rightarrow \eqref{e3.2}$.

Suppose now $M^0_\varphi f\in{L^p}$. For any given $t_{0}< 0$ and $K\in (0,\infty)$, let
\begin{align}\label{ef4.13}
\Omega_{t_{0}}^{K}:=\lf\{x\in \mathbb{B}_{K}:M_{M,\,\widetilde{M}}^{0\,(t_{0},\,L)}f(x)\leq C_{2} M_{\varphi}^{(t_{0},\,L)}f(x)\r\},
\end{align}
where $C_{2}:=2^{1/p}C_{1}$ with $C_{1}$ to be specified later and $\mathbb{B}_{K}$ is a ball as in Lemma \ref{l4.8}.
Combining Lemmas \ref{l4.6} and \ref{l4.7}, we know that there exist integer $M>0$ large enough and integer $\widetilde{M}\geq M$ such that
\begin{align}\label{e3.31x}
\lf\|M_{M,\,\widetilde{M}}^{0(t_{0},\,L)}f\r\|_{{L^{p}(\mathbb{B}_{K}})}\leq C_{1}\lf\|M_{\varphi}^{(t_{0},\,L)}f\r\|_{L^{p}(\mathbb{B}_{K})},
\end{align}
where constant $C_{1}$ is independent of $t_{0}< 0$.

Next, we claim that
\begin{align}\label{ef4.14}
\int_{\mathbb{B}_{K}} \lf[M_{\varphi}^{(t_{0},\,L)}f(x)\r]^{p}dx\leq 2 \int_{\Omega_{t_{0}}^{K}} \lf[M_{\varphi}^{(t_{0},\,L)}f(x)\r]^{p}dx.
\end{align}
Indeed, this follows from \eqref{e3.31x}, $M_{\varphi}^{(t_{0},\,L)}f\in L^{p}(\mathbb{B}_{K})$ and
\begin{align*}
\int_{\mathbb{B}_{K}/\Omega_{t_{0}}^{K}} \lf[M_{\varphi}^{(t_{0},\,L)}f(x)\r]^{p}dx &\leq C^{-p}_{2} \int_{\mathbb{B}_{K}/\Omega_{t_{0}}^{K}}\lf[M_{M,\,\widetilde{M}}^{0\,(t_{0},\,L)}f(x)\r]^{p}dx\\
&\leq (C_{1}/C_{2})^{p}\int_{\mathbb{B}_{K}} \lf[M_{\varphi}^{(t_{0},\,L)}f(x)\r]^{p}dx,
\end{align*}
where $(C_{1}/C_{2})^{p}=1/2$.

For given $t_{0}< 0$, let
\begin{align}\label{e4.13}
\Omega_{t_{0}}:=\lf\{x\in \mathbb{R}^{n}:M_{M}^{0(t_{0},\,L)}f(x)\leq C_{2} M_{\varphi}^{(t_{0},\,L)}f(x)\r\}.
\end{align}
Observe that the set $\Omega_{t_{0}}^{K}$ is monotonically increasing with respect to $K$ and
\begin{align}\label{ef3.36}
\lim_{K\rightarrow\infty}\Omega_{t_{0}}^{K}=\Omega_{t_{0}}.
\end{align}
By \eqref{ef4.14} and \eqref{ef3.36}, and letting $K\rightarrow\infty$, we obtain
\begin{align}\label{e4.14}
\int_{\mathbb{R}^{n}} M_{\varphi}^{(t_{0},\,L)}f(x)^{p}dx\leq 2 \int_{\Omega_{t_{0}}} M_{\varphi}^{(t_{0},\,L)}f(x)^{p}dx.
\end{align}

We also claim that for $0<q<p$ there exists a constant $C_{3}>0$ such that for any $t_{0}< 0$,
\begin{align}\label{e4.15}
M_{\varphi}^{(t_{0},\,L)}f(x)\leq C_{3}\lf[M_{\Theta}\lf(M_{\varphi}^{0\,(t_{0},\,L)}f\r)^{q}(x)\r]^{1/q},
\end{align}
where $M_{\Theta}$ is as in Definition \ref{d3.6}.
Indeed, let $t\geq t_{0}$, $y\in \theta(x,t)$ and
\begin{align*}
F_x(y,\,t):=|(f\ast\varphi_{x,\,t})(y)|\,(1+|M_{x,\,t_{0}}^{-1}\,y|)^{-L}(1+{2^{t+{t_{0}}}})^{-L}.
\end{align*}
Suppose $x\in \Omega_{t_{0}}$ and let $F_{l}^{*\, t_{0}}(x)$ be as in \eqref{e3.3x} with $l=0$. Then there exist $t'\in\mathbb{R}$ with $t'\geq t_{0}$ and $y' \in \theta(x,\,t') $ such that
\begin{align}\label{eab4.18}
F_x(y',\,t')\geq F^{*\,t_{0}}_{0}(x)/2=M_{\varphi}^{(t_{0},\,L)}f(x)/2.
\end{align}
Consider $x'\in y'+M_{x,\,t'+lJ}{(\mathbb{B}^{n})}$ for some integer $l\geq 1 $ to be specified later. Let $\Phi(z):=\varphi\lf( z+M^{-1}_{x,\,t'}(x'-y')\r)-\varphi(z)$. Obviously, we have
\begin{align}\label{ea4.18}
f\ast\varphi_{x,\,t'}(x')-f\ast\varphi_{x,\,t'}(y')=f\ast\Phi_{x,\,t'}(y').
\end{align}

Let us first estimate $\|\Phi\|_{\mathcal{S}_{M,\,\widetilde{M}}}$. From $x'\in y'+M_{x,\,t'+lJ}{(\mathbb{B}^{n})}$, we deduce that $$M^{-1}_{x,\,t'}(x'-y')\in M^{-1}_{x,\,t'}M_{x,\,t'+lJ}{(\mathbb{B}^{n})}.$$ By this and the mean value theorem, we obtain
\begin{align}\label{e3.37x}
\|\Phi\|_{\mathcal{S}_{M,\,\widetilde{M}}}&\leq \sup_{h\in M^{-1}_{x,\,t'}M_{x,\,t'+lJ}{(\mathbb{B}^{n})}}\|\varphi(\cdot+h)-\varphi(\cdot)\|_{\mathcal{S}_{M,\,\widetilde{M}}}\\
&=\sup_{h\in M^{-1}_{x,\,t'}M_{x,\,t'+lJ}{(\mathbb{B}^{n})}}\sup_{z\in {\mathbb{R}^{n}}}\sup_{|\alpha|\leq M}(1+|z|)^{\widetilde{M}}|(\partial^{\alpha}\varphi)(z+h)-\partial^{\alpha}\varphi(z)|\nonumber\\
&\leq C\sup_{h\in M^{-1}_{x,\,t'}M_{x,\,t'+lJ}{(\mathbb{B}^{n})}}\sup_{z\in {\mathbb{R}^{n}}}\sup_{|\alpha|\leq M+1}(1+|z|)^{\widetilde{M}}|(\partial^{\alpha}\varphi)(z+h)|\nonumber\\
& \ \ \ \times \sup_{h\in M^{-1}_{x,\,t'}M_{x,\,t'+lJ}{(\mathbb{B}^{n})}}|h|.\nonumber
\end{align}
From \eqref{e2.2}, we deduce
$$ \|M^{-1}_{x,\,t'}M_{x,\,t'+lJ}\|\leq a_{5}2^{-a_{6}lJ},$$
which implies
$$M^{-1}_{x,\,t'}M_{x,\,t'+lJ}(\mathbb{B}^{n})\subset a_{5}2^{-a_{6}lJ}\mathbb{B}^{n}.$$
By this and ${h\in M^{-1}_{x,\,t'}M_{x,\,t'+lJ}{(\mathbb{B}^{n})}}$, we have $|h|\leq a_{5}2^{-a_{6}lJ}$. From this and \eqref{e4.11}, we deduce that
\begin{align*}
1+|z|\leq (1+|z+h|)(1+|h|)\leq C(1+|z+h|), \ \ z\in \mathbb{R}^n
\end{align*}
Applying this and $|h|\leq a_{5}2^{-a_{6}lJ}$ in \eqref{e3.37x}, we obtain
\begin{align}\label{e3.29}
\|\Phi\|_{\mathcal{S}_{M,\,\widetilde{M}}}&\leq C \sup_{h\in M^{-1}_{x,\,t'}M_{x,\,t'+lJ}{(\mathbb{B}^{n})}}\sup_{z\in {\mathbb{R}^{n}}}\sup_{|\alpha|\leq M+1}(1+|z+h|)^{\widetilde{M}}|(\partial^{\alpha}\varphi)(z+h)|\\
&\hs\times\sup_{h\in M^{-1}_{x,\,t'}M_{x,\,t'+lJ}{(\mathbb{B}^{n})}}|h|\leq C\|\varphi\|_{\mathcal{S}_{M+1,\,\widetilde{M+1}}}a_{5}2^{-a_{6}lJ}\leq C_{4}2^{-a_{6}lJ},\nonumber
\end{align}
where positive constant $C_{4}$ doesn't depend on $L$.

Moreover, notice that for any $x'\in M_{x,\,t'+lJ}{(\mathbb{B}^{n})}+y'$, there exists $z\in \mathbb{B}^{n}$ such that $x'=M_{x,\,t'+lJ}z  +y'$. By \eqref{e4.11}, \eqref{e2.2} and $t'\geq t_{0}$, we have
\begin{align}\label{e3.30}
\lf(1+\lf|M^{-1}_{x,\,t_{0}}x'\r|\r)&\leq\lf(1+\lf|M^{-1}_{x,\,t_{0}}y'\r|\r)\lf(1+\lf\|M^{-1}_{x,\,t_{0}}M_{x,\,t'+lJ}\r\||z|\r)\\
&\leq\lf(1+\lf|M^{-1}_{x,\,t_{0}}y'\r|\r)\lf(1+a_{5}2^{-a_{6}(t'-t_{0}+lJ)}|z|\r)
\leq2a_{5}\lf(1+\lf|M^{-1}_{x,\,t_{0}}y'\r|\r)\nonumber.
\end{align}
Thus, for any $x\in\Omega_{t_{0}}$, from \eqref{ea4.18}, \eqref{e3.30}, \eqref{eab4.18}, \eqref{e3.29}, Lemma \ref{l4.9} and \eqref{e4.13}, it follows that
\begin{align*}
2^{L}a_{5}^{L}F_x(x',\,t')&=2^{L}a_{5}^{L}\lf[|(f\ast\varphi_{x,\,t'})(x')|\,(1+|M_{x,\,t_{0}}^{-1}\,x'|)^{-L}(1+2^{t'+t_{0}})^{-L}\r]\\
&\geq[|f\ast\varphi_{x,\,t'}(y')|-|f\ast\Phi_{x,\,t'}(y')|]
\lf(1+\lf|M^{-1}_{x,\,t_{0}}y'\r|\r)^{-L}\lf(1+2^{t'+t_{0}}\r)^{-L}\\
&\geq F_x(y',\,t')-M_{M,\,\widetilde{M}}^{(t_{0},\,L)}f(x)\|\Phi\|_{\mathcal{S}_{M,\,\widetilde{M}}}\\
&\geq M_{\varphi}^{(t_{0},\,L)}f(x)/2-C_{4}2^{-a_{6}lJ}C M_{M,\,\widetilde{M}}^{0\,(t_{0},\,L)}f(x)\\
&\geq M_{\varphi}^{(t_{0},\,L)}f(x)/2-C_{4}C_{2}C2^{-a_{6}lJ} M_{\varphi}^{(t_{0},\,L)}f(x).
\end{align*}
We choose integer $l\geq1$ large enough such that $C_{4}C_{2}C2^{-a_{6}lJ}\leq 1/4$. Therefore, for any $x\in\Omega_{t_{0}}$ and $x'\in M_{x,\,t'+lJ}{(\mathbb{B}^{n})}+y'$, we further have
\begin{align}\label{ea4.20}
 2^{L}a_{5}^{L}F_x(x',\,t')\geq M_{\varphi}^{(t_{0},\,L)}f(x)/2-C_{4}C_{2}C2^{-a_{6}lJ} M_{\varphi}^{(t_{0},\,L)}f(x)
\geq M_{\varphi}^{(t_{0},\,L)}f(x)/4.
\end{align}

Besides, by $y' \in \theta(x,\,t')$ and Proposition \ref{l2.2}, we have
\begin{align}\label{e3.32}
&M_{x,\,t'+lJ}{(\mathbb{B}^{n})}+y'\subseteq M_{x,\,t'+lJ}{(\mathbb{B}^{n})}+M_{x,\,t'}(\mathbb{B}^{n})+ x\\
&\quad\subseteq 2M_{x,\,t'}{(\mathbb{B}^{n})}+ x\subseteq \theta(x,\,t'-J)\nonumber.
\end{align}
Thus, for any $x\in\Omega_{t_{0}}$ and $t\geq t_{0}$, by \eqref{ea4.20} and \eqref{e3.32}, we obtain
\begin{align*}
\lf[M_{\varphi}^{(t_{0},\,L)}f(x)\r]^{q}&\leq\frac{4^{q}2^{Lq}a_{5}^{Lq}}{|M_{x,\,t'+lJ}{(\mathbb{B}^{n})}|}\int_{y'+M_{x,\,t'+lJ}{(\mathbb{B}^{n})}}[F_x(z,t')]^{q}dz\\
&\leq C4^{q}2^{Lq}a_{5}^{Lq}\frac{2^{(l+1)J}}{\lf|\theta(x,\,t'-J)\r|}\int_{\theta(x,\,t'-J)}\lf[M_{\varphi}^{0\,(t_{0},\,L)}f(z)\r]^{q}dz\\
&\leq C_{3}M_{\Theta}\lf(\lf(M_{\varphi}^{0\,(t_{0},\,L)}f\r)^{q}\r)(x),
\end{align*}
which shows the above claim \eqref{e4.15}.

Consequently, by \eqref{e4.14}, \eqref{e4.15} and Proposition \ref{l4.3} with $p/q>1$, we have
\begin{align}\label{e4.16}
\int_{\mathbb{R}^{n}} \lf[M_{\varphi}^{(t_{0},\,L)}f(x)\r]^{p}dx&\leq 2 \int_{\Omega_{t_{0}}} \lf[M_{\varphi}^{(t_{0},\,L)}f(x)\r]^{p}dx\\
&\leq 2{C_{3}}^{p} \int_{\Omega_{t_{0}}}\lf[M_{\Theta}\lf(\lf(M_{\varphi}^{0(t_{0},\,L)}f\r)^{q}\r)(x)\r]^{p/q}dx\nonumber\\
&\leq C_{5} \int_{\mathbb{R}^{n}} \lf[M_{\varphi}^{0(t_{0},\,L)}f(x)\r]^{p}dx\nonumber,
\end{align}
where the constant $C_{5}$ depends on $p/q>1$, $L\geq0$ and $p(\Theta)$, but is independent of $t_{0}<0$. This inequality is crucial as it gives a bound of the non-tangential by the radial maximal function in $L^{p}$. The rest of the proof is immediate.

For any $x\in\mathbb{R}^{n}$, $y\in\mathbb{R}^{n}$ and $t<0$, by \eqref{e2.2}, we obtain
\begin{align*}
\lf|M_{x,\,t}^{-1}y\r|&=\lf|M_{x,\,t}^{-1}M_{x,\,0}M_{x,\,0}^{-1}y\r|\leq\lf\|M_{x,\,t}^{-1}M_{x,\,0}\r\|\lf\|M_{x,\,0}^{-1}\r\||y|\\
&\leq a_{5}2^{a_{6}t}\lf\|M_{x,\,0}^{-1}\r\||y|\rightarrow 0 \ \  \text{as} \ \ t\rightarrow -\fz.\nonumber
\end{align*}
Hence, we obtain $M_{\varphi}^{(t_{0},\,L)}f(x)$ converges pointwise and monotonically to $M_{\varphi}f(x)$ for all $x\in{\mathbb{R}^{n}}$ as $t_{0}\rightarrow{-\fz}$, which, together with \eqref{e4.16} and the monotone convergence theorem, further implies that $M_{\varphi}f\in L^{p}$. Therefore, we can now choose $L=0$ and again by \eqref{e4.16}  and the monotone convergence theorem, we have $\|M_{\varphi}f\|^{p}_{p}\leq C_{5}\|M^{0}_{\varphi}f\|^{p}_{p}$, where $C_{5}$ corresponds to $L=0$ and is independent
of $f\in\mathcal{S'}$. This finishes the proof of Theorem \ref{t4.1}.
\end{proof}

\bigskip

\noindent Aiting Wang, Wenhua Wang, Xinping Wang and Baode Li (Corresponding author)

\medskip

\noindent College of Mathematics and System Science, Xinjiang University, Urumqi, 830046,
P. R. China
\smallskip

\noindent{\it E-mail address}: \texttt{2358063796@qq.com}\quad(Aiting Wang)

\hspace{2.32cm}\texttt{1663434886@qq.com}\quad(Wenhua Wang)

\hspace{2.32cm}\texttt{569536403@qq.com}\quad(Xinping Wang)

\hspace{2.32cm}\texttt{baodeli@xju.edu.cn}\quad(Baode Li)

\end{document}